\documentclass{article}
\newtheorem{theorem}{Theorem}[section]

\newtheorem{lemma}[theorem]{Lemma}
\newtheorem{corollary}[theorem]{Corollary}

\newtheorem{remark}[theorem]{Remark}
\def\whitebox{{\hbox{\hskip 1pt
 \vrule height 6pt depth 1.5pt
 \lower 1.5pt\vbox to 7.5pt{\hrule width
    3.2pt\vfill\hrule width 3.2pt}%
 \vrule height 6pt depth 1.5pt
 \hskip 1pt } }}
\def\qed{\ifhmode\allowbreak\else\nobreak\fi\hfill\quad\nobreak
     \whitebox\medbreak}
\newcommand{\proof}{\noindent{\it Proof.}\ }

\newcommand{\ignore}[1]{}
\usepackage{amsmath}
\usepackage{amssymb}
 \usepackage{tikz}

\usepackage{amssymb}
\usepackage{mathrsfs}
\usepackage{}
\usepackage{amsfonts}
 \usepackage{color}
\usepackage{braket}
 \usepackage{tikz}

\title{The Existence of Diagonal Quantum Latin Squares with Maximum Cardinality
\thanks{
Supported by NSFC Grant No. 12271390.

\ $^\dag$Corresponding author: Yang Li (liyang$\_$lll@163.com).
}
}

\author{\small  Lin Huang and Yang Li$^{\dag}$\\
 \small Department of Mathematics,
  Soochow University,
 Suzhou 215006, P. R. China
}
\date{}

\begin{document}

\maketitle

\begin{abstract}
\medskip
A quantum Latin square of order \(n\), denoted by \(\operatorname{QLS}(n)\), is an \(n \times n\) square whose entries are unit column vectors in the \(n\)-dimensional Hilbert space \(\mathcal{H}_n\), such that each row and each column forms an orthonormal basis of \(\mathcal{H}_n\). The cardinality of a QLS($n$) is the number of distinct  vectors
up to a global phase
in the array.
A \(\mathrm{QLS}(n)\) whose main diagonal and anti-diagonal each forms an orthonormal basis of \(\mathcal{H}_n\) is called a diagonal quantum Latin square (\(\mathrm{DQLS}(n)\)).
In this paper, we focus on the existence of
the \(\mathrm{DQLS}(n)\) with maximum cardinality ($\operatorname{MCDQLS}(n)$). 
By  employing direct constructions based on 
row-quantum Latin rectangle and special  complete mapping, 
together with the recursive techniques such as the singular direct product construction, 
We have almost completely determined the existence of \(\mathrm{MCDQLS}(n)\), except for a few exceptional cases. This result is based on the study of the existence of idempotent \(\mathrm{QLS}(n)\) with maximum cardinality (\(\mathrm{MCQLS}(n)\)), and implies an existence result for pandiagonal quantum Latin squares with maximum cardinality (\(\mathrm{MCPQLS}(n)\)).

\


\smallskip
\end{abstract}

\section{Introduction}

A Latin square of order \(n\), denoted by \(\mathrm{LS}(n)\), is an \(n \times n\) square with entries from the \(n\)-element symbol set \(\{0,1,\dots,n-1\}\) such that each row and each column forms a permutation of the symbol set. With the development of quantum information theory, some classical combinatorial objects in discrete mathematics have been successfully quantized.
As a fundamental structure in combinatorial design theory, Latin squares were generalized to quantum Latin squares by Musto and Vicary \cite{cite1} in 2016. They further showed that these novel structures can be used to construct unitary error bases (UEBs).

A quantum Latin square of order $n$,  denoted by $\operatorname{QLS}(n)$,  is an $n \times n$ square whose entries are unit column vectors from the $n$-dimensional Hilbert space $\mathcal{H}_n$,  such that each row and column forms an orthonormal basis of $\mathcal{H}_n$. Subsequently, quantum Latin squares further attracted the attention of researchers.

 In 2017, 
Musto \cite{Musto2017} introduced  the notions of weakly orthogonal and orthogonal quantum Latin squares, which proved 
crucial for constructing mutually unbiased bases (MUBs). In 2021, Nechita and Pillet \cite{Nechita21} proposed the concept of quantum Sudoku, which is a special case of quantum Latin squares. 
In the same year, Paczos et al. \cite{Paczos21}  introduced the cardinality measure for quantum Latin squares, proving the existence of quantum Sudoku with maximum cardinality and deriving new MUBs. They also showed that the possible cardinality of QLS$(2)$ is $2$, and that of QLS$(3)$ is $3$.

A QLS($n$) can be obtained by replacing each entry \( i \in \{0, 1, \dots, n-1\} \) in a classical Latin square with the computational basis vector \( \ket{i} \in \mathcal{H}_n \). A QLS($n$) is called \textit{classical} if all entries are constrained to the computational basis \(\{\ket{0}, \ket{1}, \dots, \ket{n-1}\}\).
In quantum theory, two unit vectors \(\ket{u}, \ket{v} \in \mathcal{H}_n\) are regarded as identical (\(\ket{u}=\ket{v}\)) if they have the same global phase, i.e., there exists a real number \(\theta\) such that \(\ket{u} = e^{i\theta}\ket{v}\); otherwise, they are considered distinct (\(\ket{u}\neq\ket{v}\)). The cardinality \(c\) of a QLS($n$) is the number of distinct (up to a global phase) vectors in the array. Clearly the cardinality \(c\) of a QLS($n$) satisfies that \(n \leq c \leq n^2\). Furthermore, in this paper, by a QLS($n$) of maximum cardinality, denoted by MCQLS($n$), we mean that its cardinality is \(c = n^2\); i.e., it is of full cardinality.

The cardinality problem of quantum Latin squares has further attracted research interest. During the 2025-2026 period, new results continue to appear.
 Zhang et al. \cite{Cao-July} determined the existence of  MCQLS$(n)$ with $n \geq 9\times 77^4+4$.  Shortly after their results, Zhang et al. \cite{cite16}  constructed QLS(4m)s with
all possible cardinalities for $m \geq 2$.
Zhang and Cao \cite{cite18}  further confirmed that MCQLS$(n)$ exist for all \(n \geq 4\) with  $11$ possible exceptions.   Zang et al.\cite{cite19,cite19-2}  completely resolved the existence of QLS($n$) with maximum cardinality for $n \geq 4$, and also explored some possible cardinality range of a
QLS$(n)$ for any $n\geq 4$.
Then Zhang and Ji~\cite{cite17} and Xu~\cite{Xu26}  studied the construction of QLS($6m$) with all possible cardinalities.
Some works on quantum Latin squares (e.g., Refs.~\cite{Cao-July} and~\cite{cite18}) involve the concept of a transversal, which has always been a very important concept in classical Latin square theory.

A transversal of a Latin square of order \(n\) is a set of \(n\) cells, no two in the same row or column, containing exactly one occurrence of each symbol from the symbol set. A Latin square is called an idempotent Latin square if its main diagonal is a transversal, and a diagonal Latin square, denoted by $\mathrm{DLS}(n)$, if both its main diagonal and its anti-diagonal are transversals.
Furthermore, a Latin square \(\mathrm{LS}(n)\) is called a pandiagonal Latin square, denoted \(\mathrm{PLS}(n)\), if its main diagonal, anti-diagonal, and all broken diagonals in both directions are each a  transversal. Two transversals in a LS$(n)$ are called disjoint if they have no overlapping cells. \(\mathrm{PLS}(n)\) is sometimes also called a Knut Vik design (see Refs.~\cite{hedayat1977complete, Bell07} for example).
After diagonal Latin squares, the existence of Latin squares with disjoint transversals is also an interesting research problem. For example, if a Latin square of order \(n\) has \(n\) disjoint transversals if and only if two   mutually orthogonal Latin squares exist. Here, 
Two Latin squares $L_1 = (a_{ij})$ and $L_2 = (b_{ij})$ of order $n$ are said to be orthogonal if the $n^2$ ordered pairs $(a_{ij}, b_{ij})$ are all distinct.

For the quantum Latin square version, a transversal in a QLS$(n)$ refers to a set of $n$ cells, no two in the same row or column, the vectors in these cells forming an orthonormal  basis of $\mathcal{H}_n$. Two transversals in a LS$(n)$ are called disjoint if they have no overlapping cells.
A \(\mathrm{QLS}(n)\) is called an idempotent quantum Latin square if its main diagonal is a transversal.
A \(\mathrm{QLS}(n)\) is called a diagonal quantum Latin square, denoted by \(\mathrm{DQLS}(n)\) , if both its main diagonal and its anti-diagonal are transversals.
A \(\mathrm{QLS}(n)\) is called a pandiagonal quantum Latin square, denoted as \(\mathrm{PQLS}(n)\), if its main diagonal, anti-diagonal, and all broken diagonals in both directions are each a transversal. In the context of cardinality, a \(\mathrm{DQLS}(n)\) or a \(\mathrm{PQLS}(n)\) of maximum cardinality is denoted by \(\mathrm{MCDQLS}(n)\) or \(\mathrm{MCPQLS}(n)\), respectively.

For the study of quantum Latin squares and even of quantum theory, many references can be consulted, e.g., \cite{Paczos21, rather2022dual, rather2022thirty, rather2023absolutely, reuttera2019biunitary, zang2022further, zang2021quantum, zang2023quantum, zang2022mutually, zang2019uniform, zuo2021entanglement, zyczkowski2023}, and so on.
However,
this paper focuses on the existence and constructions of MCDQLS$(n)$, which need be based on the existence and constructions of idempotent  MCQLS$(n)$,  and also implies an existence result for MCPQLSs. During the exploration, we will also pay attention to the constructions of pandiagonal quantum Latin squares.
In fact, we attempt to carry over some results from the classical theory of Latin squares to the quantum Latin square version in a parallel manner. At the end of this section, we list the following known results as the foundation of this study.

\begin{theorem}\label{DLS-existence}\cite{cite21}
\begin{enumerate}
    \item[$(1)$] For any positive integer $n$ with $n\neq 2$, there exists an idempotent $\mathrm{LS}(n)$. 
    \item[$(2)$]  For any positive integer $n$ with $n\neq 2,3$, there exists a $\mathrm{DLS}(n)$. 
    \end{enumerate}
\end{theorem}

\begin{lemma}\cite{cite21} \label{lem2.2.5} For any positive integer \(n\) with \(n \neq 2, 6\), there exists an \(\mathrm{LS}(n)\) that has \(n\) pairwise disjoint transversals, one of which is the main diagonal. 
\end{lemma}

\begin{lemma} \cite {hedayat1977complete} \label{lem-pls}
    A $\mathrm{PLS}(n)$ exists if and only if $\gcd(n,6)=1$.

\end{lemma}

\begin{theorem}\cite{cite19}\label{MCQLS(n)}
    For any integer $n \geq 4$, there exists an $\mathrm{MCQLS}(n)$.
\end{theorem}

\begin{lemma}\cite{cite20,Paczos21} \label{lem1.2.2}
    If $n \in \{2, 3, 4, 5\}$, there does not exist a  non-classical idempotent  $\mathrm{DQLS}(n)$. 
\end{lemma}

\begin{lemma}\cite{Cao-July} \label{lem1.2.4} 
    For any positive interge $n$, there exists an $\mathrm{MCQLS}(n^4)$ containing $n^4$ disjoint transversals. 
\end{lemma}

The rest of the paper is organized as follows. Section $2$ focuses on the direct constructions for MCQLS$(n)$ using row-quantum Latin rectangle and special complete
mapping. Section $3$  mainly introduces a singular direct product construction for MCDQLS$(n)$, and makes some preparations for the implementation of the recursive methods. Section $4$ resolves the existence of MCDQLS$(n)$.
Section~$5$ then summarizes our new results on the existence of MCQLS$(n)$, together with the results for idempotent MCQLS$(n)$ and MCPQLS$(n)$.

\section{Direct Constructions}

In this section, we present two types of direct construction for MCDQLS$(n)$: one based on the row-quantum Latin rectangle, and the other based on a special complete mapping.

\subsection{Row-quantum Latin rectangle}

An $m \times n$ row-quantum Latin rectangle is an array with $m$ rows and $n$
columns whose entries are unit vectors from $\mathcal{H}_n$ such that the entries in each row forms an
orthonormal basis.
If the vectors in the same column of an $m \times n$ row-quantum rectangle are mutually
orthogonal, then it is called an $m\times n$ quantum Latin rectangle. It is important to note that
such a rectangle must satisfy $m \leq n$. 
The
cardinality (number of distinct unit vectors up to a global phase) of such a rectangle satisfies $n \leq c \leq mn$. If \( c = mn \), then \( c \) is called the maximum.

Let $U = (\ket{u_{i,k}})$ and $V = (\ket{v_{j,l}})$ denote an $m \times n$ row-quantum Latin rectangle
with cardinality $c_1$ in $\mathcal{H}_n$ and $n \times m$ row-quantum Latin rectangle with cardinality $c_2$
in $\mathcal{H}_m$, respectively. Let $W$ be an array of order $mn$, which is divided into $mn$ blocks of
size $n \times m$.
Let $\ket{w_{i,j,k,l}}$ denote the element located in the $k$-th row and $l$-th column of the $(i,j)$-th block in $W$, and define it as the tensor product $\ket{w_{i,j,k,l}} = \ket{u_{i,j+k}}\otimes \ket{v_{j,i+l}}$, where $j+k$ is taken modulo $n$ and $i+l$ is taken modulo $m$. Then $W$ is a QLS$(mn)$ with cardinality \( c_1c_2 \).

\begin{lemma} \cite{Cao-July}\label{construction-row-quantum Latin rectangle}
If there exists an \( m \times n \) row-quantum Latin rectangle with cardinality \( c_1 \), and an \( n \times m \) row-quantum Latin rectangle with cardinality \( c_2 \), then there exists a \( QLS(mn) \) with cardinality \( c_1c_2 \).
\end{lemma}

The proof of Lemma~\ref{construction-row-quantum Latin rectangle} in \cite{Cao-July} relies on the following important fact, which is also frequently used in our subsequent discussion of cardinality.

\begin{lemma} \cite{Cao-July}\label{state vector}   Let \( \ket{a}, \ket{b} \in \mathcal{H}_n \), and \(\ket{c}, \ket{d} \in \mathcal{H}_n \) be unit vectors. Then \( \ket{a} \otimes \ket{c} \) and \( \ket{b} \otimes \ket{d} \) are identical if and only if \( \ket{a} \) is identical to \( \ket{b} \) and \( \ket{c} \) is identical to \( \ket{d} \).
\end{lemma}

Sometimes, the MCQLS constructed by Lemma~\ref{construction-row-quantum Latin rectangle} satisfies some transversal properties. Lemma \ref{lem1.2.4} is precisely such a result.
The proof of Lemma \ref{lem1.2.4} implies the following simple  but important facts. Let $U=(\ket{u_{i,k}})_{0\leq i,j\leq n-1}$ and $V=(\ket{v_{j,l}})_{0\leq i,j\leq n-1}$ be two $n \times n$ row-quantum Latin rectangle with maximum cardinality. If the $n$ vectors in some column of  $V$, the $l$-th column for example,  are  mutually orthogonal, then the $n^2\times n^2$ array $W$ defined  above
contains $n$  disjoint transversals
$$
T_{kl} =\{ \ket{ u_{i,k}} \otimes \ket{ u_{j,l} } \mid i, j = 0, 1, \dots, n - 1 \}, k=0,1,\dots, n-1,
$$
where
 the entries of $T_{kl}$ occupy the cell set  $\{(in+(k-j),jn +(l-i)), i, j = 0, 1, \dots, n - 1\}$ of $W$. Note that $k-j$ and $l-i$ are taken modulo $n$ here.
 Furthermore,  if  the $n$ vectors in the  $l'$-th column ($l'\neq l$) of  $V$ are  also mutually orthogonal, then  $$
T_{kl'} =\{ \ket{ u_{i,k}} \otimes \ket{ u_{j,l'} } \mid i, j = 0, 1, \dots, n - 1 \}, k=0,1,\dots, n-1
$$ are also  $n$ disjoint  transversals with $T_{kl'}$ having  cell set $\{(in+(k-j),jn +(l'-i)), i, j = 0, 1, \dots, n - 1\}$, Clearly, $T_{kl}$, $T_{kl'}$, $k=0,1,\dots, n-1$
form $2n$ disjoint  transversals.
Then, by Theorem~\ref{MCQLS(n)}, Lemma \ref{lem1.2.4} can be improved as follows.

\begin{lemma}\label{MCIDQLS(v^2)} For any positive integer $n\geq4$, there exists an MCQLS$(n^2)$ having $n^2$ disjoint transversals.

\end{lemma}

In fact, when $n$ is even and the above $V$ has two orthogonal columns, we can adjust the rows and columns of $W$ to transform it into an $\mathrm{MCDQLS}(n^2)$.

\begin{lemma} Suppose $n$ is an even integer. If there exists an $n \times n$ row-quantum Latin rectangle with maximum cardinality such that two of its columns form orthonormal bases of $\mathcal{H}_n$, then there exists an  MCDLS$(n^2)$.

\end{lemma}

\proof Let $U=(\ket{u_{i,k}} )$  and $V=(\ket{v_{j,l}} )$  be both an $n \times n$ row-quantum Latin rectangle with maximum cardinality,  where $i,j,k,l\in  \{0,1,\dots,n-1\}$.
Without loss of generality, suppose that the entries of the $0$-th column and the $\frac{n}{2}$-th column of $V$ form orthonormal bases of $\mathcal{H}_n$, respectively.

Construct an array $W=(\widetilde{w}_{s,t})$ of order $n^2$, which is divided into $n^2$ blocks of size $n \times n$.
 Define $\ket{\widetilde{w}_{in+k,jn+l}} = \ket{w_{i,j,k,l}} = \ket{u_{i,j+k}}\otimes \ket{v_{j,i+l}}$ as the entry in the $k$-th row and $l$-th column of the $(i,j)$-th block of the array $W$, where $i,j,k,l \in  \{0, 1, \dots, n - 1\}$.  We already know that $W$ is a $\mathrm{MCQLS}(n^2)$  and that $T_{k,l} = \{\ket{u_{i,k}} \otimes \ket{v_{j,l}} \mid i,j=0,1,\dots,n-1\}$ forms a transversal of $W$, where $(k,l) = (0,0)$ or $(0,\frac{n}{2})$.
To better illustrate the adjustment of $W$ below, we refer to the $(in+k, jn+l)$ cell of $W$ as the $(i,k)$-th row and $(j,l)$-th column, abbreviated as the $(i,j,k,l)$ cell, and let $\oplus_n$ and $\ominus_n$ denote addition and subtraction in $\mathbb{Z}_n$, respectively.
 Thus,
\(
\ket{w_{i,j,k,l}} = \ket{u_{i,\,j\oplus_n k}} \otimes \ket{v_{j,\,i\oplus_n l}}
\)
is equivalent to
\(
\ket{w_{i,j,\,k\ominus_n j,\,l\ominus_n i}} = \ket{u_{i,k}} \otimes \ket{v_{j,l}}.
\)

Construct a new empty array $W'$ with exactly the same labeling and block structure as $W$. Rearrange the columns of $W$ in a suitable order to form the columns of $W'$ so that the entry of $T_{0,0}$ in $W$ lies on the main diagonal of the new array $W'$.
Note that the  $(i',j',0\ominus_n j',\frac{n}{2}\ominus_n i')$ cell in the cell set of $T_{0,\frac{n}{2}}$ and the  $((i'\oplus_n\frac{n}{2}),j',0\ominus_n j',0\ominus_n(i'\oplus_n\frac{n}{2}))$ cell in the cell set of $T_{0,0}$ are on the same column of $W$. So, when the entry in cell $(i'\oplus_n\frac{n}{2},j',0\ominus_n j',0\ominus_n(i'\oplus_n\frac{n}{2}))$ of $T_{0,0}$ is moved onto the main diagonal of $W'$,  the entry in the $(i',j',0\ominus_n j',\frac{n}{2}\ominus_ni')$ cell of $T_{0,\frac{n}{2}}$ is moved to the
$(i',i'\oplus_n\frac{n}{2},0\ominus_n j',0\ominus_nj')$  cell of $W'$. 

Considering the  $(i',i'\oplus_n\frac{n}{2},0\ominus_n j',0\ominus_nj')$ cell and the  $(i',i',0\ominus_n j',0\ominus_nj')$ cell  (\(i', j' \in \mathbb{Z}_n\)
),
it is clear that $W'$ has the other transversal whose cell set can be obtained by moving the main diagonal cells $\frac{n^2}{2}$ columns to the right.

Starting from $W'$, an MCDQLS$(n^2)$ is obtained by first flipping the right half of the columns and then flipping the lower half of the rows.

\qed

\begin{theorem} \label{DQLS-square} For any even integer $n\geq 4$, there exists an  MCDQLS$(n^2)$.
\end{theorem}

\subsection{Complete mapping}

In 2025, Zang  et al. \cite{cite19}  employed Fourier matrix to construct an MCQLS\((v)\) for any order \( v \geq 4 \) (see  Theorem~\ref{MCQLS(n)}). Building on their approach, this subsection presents a direct construction of $\mathrm{MCDQLS}(v)$ via a strong complete mapping that satisfies a specific condition.

The construction of MCQLS\( (v)\) given by  Zang et al. \cite{cite19} utilizes the properties of the \(v\)-th primitive roots of unity, and the underlying idea can be interpreted from the perspective of linear transformations. Consider the Fourier matrix
\begin{align*}
            F_v= \frac{1}{\sqrt{v}}
\begin{pmatrix}
    1 & 1 & \dots & 1 & \dots & 1\\
    1 & \omega & \dots & \omega^j & \dots & \omega^{v-1}\\
    1 & \omega^2 & \dots & \omega^{2j} & \dots & \omega^{2(v-1)}\\
    \vdots & \vdots & & \vdots &  & \vdots\\
    1 & \omega^{k} & \dots & \omega^{kj} & \dots & \omega^{k(v-1)}\\
    \vdots & \vdots & & \vdots &   & \vdots\\
    1 & \omega^{v-1} & \dots & \omega^{(v-1)j} & \dots & \omega^{(v-1)^2}
\end{pmatrix}, 
        \end{align*}
where $\omega$ is  a complex primitive $v$-th root of unit. Denote the column vectors of this unitary matrix by \( f_0, f_1, \dots, f_{v-1} \) in order. Clearly, they form an orthonormal basis of \(\mathcal{H}_{v}\).
Let $\mu$ be a permutation on \( \mathbb{Z}_n\) and let  \(g_i = \frac{1}{\sqrt{v}} (\omega^{\mu(0)\times i}, \omega^{\mu(1)\times i}, \dots, \omega^{\mu(v-1)\times i})^{\top}\), \(i = 0,1,\dots, v-1\). Similarly, \(g_0, g_1, \dots, g_{v-1}\)
 form another orthonormal basis of \(\mathcal{H}_{v}\).
 Then associate \(f_0, f_1, \dots, f_{v-1}\) and \(g_0, g_1, \dots, g_{v-1}\) with the unitary matrices \(F_0, F_1, \dots, F_{v-1}\) and \(G_0, G_1, \dots, G_{v-1}\), respectively, where 
 \(
F_j= \operatorname{diag}(\omega^{0\times j}, \omega^{1 \times j}, \dots, \omega^{(v-1) \times j})
\) (a diagonal matrix with diagonal entries $\omega^{0\times j}, \omega^{1 \times j}, \ldots, \omega^{(v-1)\times j}$) 
and 
\(G_i = \operatorname{diag}( \omega^{\mu(0) \times i}, \omega^{\mu(1) \times i}, \dots, \omega^{\mu(v-1) \times i} ).\)
Suppose that the matrices of the unitary transformations \(\mathcal{F}_1, \mathcal{F}_2, \dots, \mathcal{F}_{v-1}, \mathcal{G}_1, \mathcal{G}_2, \dots, \mathcal{G}_{v-1}\) with respect to the standard basis \(\ket{0}, \ket{1}, \dots, \ket{v-1}\) of \ \(\mathcal{H}_{v}\) are exactly \(F_1, F_2, \dots, F_{v-1}, G_1, G_2, \dots, G_{v-1}\), respectively.
When constructing such a QLS$(v)$ $M$, we attempt to place at its \((i, j)\) cell the vector
\[ \ket{\phi_{i,j}}=\frac{1}{\sqrt{v}} 
(\omega^{0\times j+i\mu(0)},\dots,\omega^{t\times j+i\mu(t)},\dots,\omega^{(v-1)\times j+i\mu(v-1)})^{T},\] which is obtained 
by applying the unitary transformation \(\mathcal{F}_j\) to the vector \(g_i\). 
Equivalently, it can also be regarded as the vector obtained by applying the unitary transformation  \(\mathcal{G}_i\) to the vector \(f_j\). 

Under this construction, the \(i\)-th row of \(M\) is obtained by applying the unitary transformation \(\mathcal{G}_i\) to \(f_0, f_1, \dots, f_{v-1}\), and the \(j\)-th column of \(M\) is obtained by applying the unitary transformation \(\mathcal{F}_j\) to \(g_0, g_1, \dots, g_{v-1}\). Since both \(f_0, f_1, \dots, f_{v-1}\) and \(g_0, g_1, \dots, g_{v-1}\) are orthonormal bases of \(\mathcal{H}_{v}\), the desired property of a QLS$(v)$ is clearly guaranteed. 
Based on this, we wish to impose further restrictions on the permutation \(\sigma\) to ensure that \(M\) has maximum cardinality and satisfies the idempotent or diagonal property.

Let $(G, +)$  be an additive  group. A function \(\mu\colon G \rightarrow G\) is called a complete mapping of \(G\) if both \(\mu\) and \(\mu + \mathrm{id}\) 
(where \(\mathrm{id} \) is the identity permutation and \((\mu + \mathrm{id})(t) = \mu(t) + t\))
are permutations of \(G\). 
Another closely related concept is that of an orthomorphism of \((G, +)\), which is a permutation \(\mu: G \rightarrow G\) such that \(\mu - \mathrm{id}\) is also a permutation. A complete mapping of \(G\) is called a strong complete mapping if it is also an orthomorphism of \((G, +)\).
We say that a permutation \(\mu\) of \(\mathbb{Z}_v\) satisfies the condition  $(*)$ if there exist distinct \(s, t, k \in \mathbb{Z}_v\) such that 
\begin{align}
\gcd((t-s)(\mu(k)-\mu(s)) - (k-s)(\mu(t)-\mu(s)),\, v ) = 1. \tag{$*$}
\end{align}
Here, if \(v\) is prime, the  condition \((*)\)  reduces to the existence of distinct \(s, t, k \in \mathbb{Z}_v\) such that $(t-s)(\mu(k)-\mu(s)) - (k-s)(\mu(t)-\mu(s))\not\equiv 0 \pmod{v}$.

\begin{theorem}\label{Construction--complete mapping} If there exists a complete mapping of \(\mathbb{Z}_v\) satisfying the  condition  $(*)$, then there exists an idempotent  MCQLS$(v)$.
Furthermore, if there exists a strong complete mapping of \(\mathbb{Z}_v\) satisfying the  condition $(*)$, then there exists an MCDQLS$(v)$.
\end{theorem}
\begin{proof}
Let $M=(\ket{\phi_{i,j}})_{0\leq i,j \leq v-1}$ with \[ \ket{\phi_{i,j}}=\frac{1}{\sqrt{v}} 
(\omega^{0\times j+i\mu(0)},\dots,\omega^{t\times j+i\mu(t)},\dots,\omega^{(v-1)\times j+i\mu(v-1)})^{T},\] where $\omega =e^{\frac{2\pi \mathrm{i}}{v}}$ and
$\mu$ is a complete mapping of $Z_v$ satisfying the condition $(*)$ .
Clearly every $\ket{\phi_{i,j}}$ ($0\leq i,j \leq v-1$) is a unit vector of  $\mathcal{H}_v$. Since
$1+\omega^{j'-j}+\omega^{2(j'-j)}+\dots +\omega^{(\frac{v}{\gcd(v,j'-j)}-1)(j'-j)}=0$,
the equation 
\[
\langle \phi_{i,j} \mid \phi_{i,j'} \rangle = \frac{1}{v} \sum_{t=0}^{v-1} \omega^{t(j'-j)} = 0
\]
holds for any \(i, j, j' \in \{0,1,\dots, v-1\}\) with \(j \neq j'\). Similarly, since $\mu$ is a permutation of $Z_v$, 
 the equation
\[
\langle \phi_{i,j} \mid \phi_{i',j} \rangle = \frac{1}{v} \sum_{t=0}^{v-1} \omega^{(i'-i)\mu(t)} = 0
\]
also holds for any \(i, i', j \in \{0,1,\dots, v-1\}\) with \(i \neq i'\). Thus, $M$ is a QLS$(v)$. 

As \(\mu\) satisfies the \((*)\) condition, there exist distinct \(s, t, k \in \mathbb{Z}_v\) such that
$$((t-s)(\mu(k)-\mu(s))-(k-s)(\mu(t)-\mu(s)),v)=1.$$
Then the following system of linear congruences has only the zero solution:
$$
\begin{cases}
(t-s) x+(\mu(t)-\mu(s))y\equiv 0 \pmod{v},\\
(k-s) x+(\mu(k)-\mu(s))y\equiv 0 \pmod{v}.\\
\end{cases}
$$
Thus, the two congruences  $$(t-s)\times j+i(\mu(t)-\mu(s))\equiv (t-s)\times j'+i'(\mu(t)-\mu(s))\pmod{v}$$ and $$(k-s)\times j+i(\mu(k)-\mu(s))\equiv (k-s)\times j'+i'(\mu(k)-\mu(s))\pmod{v}$$  cannot hold simultaneously 
for any \((i, j)\neq (i', j')\). 
This implies that  
$$(
1, \omega^{(t-s)\times j+i\mu (t)-i\mu (s)}, \omega^{(k-s)\times j+i\mu (k)-i\mu (s)}
) \neq (
1,\omega^{(t-s)\times j'+i'\mu (t)-i'\mu (s)}, \omega^{(k-s)\times j'+i'\mu (k)-i'\mu (s)}
) $$ or
$$(\omega^{s\times j+i\mu(s)},\omega^{t\times j+i\mu(t)},\omega^{k\times j+i\mu(k)})\neq e^{\lambda \mathrm{i}}(\omega^{s\times j'+i'\mu(s)},\omega^{t\times j'+i'\mu(t)},\omega^{k\times j'+i'\mu(k)}) $$ for any $\lambda\in [0, 2\pi)$ and  \((i, j)\neq (i', j')\). 
Then $\ket{\phi_{i,j}})$ and $\ket{\phi_{i',j'}})$ are distinct if \((i, j)\neq (i', j')\), i.e., the cardinality of $M$  is the maximum.

By examining the diagonal entries of  $M$, we have that for any $i \neq i'$,
\[ \langle \phi_{i,i} | \phi_{i',i'} \rangle
=\frac{1}{v}\sum\limits_{t=0}^{v-1}\omega^{ti'+i'\mu (t)-(ti+i\mu (t))}
=\frac{1}{v}\sum\limits_{t=0}^{v-1}\omega^{(i'-i)(\mu (t)+t)} =0. 
\]
This orthogonality follows because   $\mu+\mathrm{id}$ is a permutation of $\mathbb{Z}_v$.
So  $M$ is an idempotent  MCQLS$(v)$.

Further more,  suppose that $\mu$ is a strong complete mapping of \(\mathbb{Z}_v\).  Similarly, by examining the anti-diagonal entries of  $M$, we have that for any $i \neq i'$, 
\[ \langle \phi_{i,v-1-i} | \phi_{i',v-1-i'} \rangle
=\frac{1}{v}\sum\limits_{t=0}^{v-1}\omega^{t(v-1-i')+i'\mu (t)-(t(v-1-i)+i\mu (t))}
=\frac{1}{v}\sum\limits_{t=0}^{v-1}\omega^{(i'-i)(\mu (t)-t)} =0.
\] This orthogonality follows because  $\mu-\mathrm{id}$ is also a permutation of $\mathbb{Z}_v$.
So  $M$ is an  MCDQLS$(v)$.
\qed

\end{proof}

In the above construction, for any $s \in \mathbb{Z}_v$, define
\[
T_s = \{\ket{\phi_{i,j}} \mid i - j \equiv s \pmod{v}\}.
\]
Since $\mu + \mathrm{id}$ is a permutation of $\mathbb{Z}_v$, we have
\[
 \langle \phi_{i,i+s} \mid \phi_{i',i'+s} \rangle
=\frac{1}{v}\sum\limits_{t=0}^{v-1}\omega^{t(i'+s)+i'\mu (t)+(-t(i+s)-i\mu (t))}
=\frac{1}{v}\sum\limits_{t=0}^{v-1}\omega^{(i'-i)(\mu (t)+t)} =0,
\]
so each $T_s$ is a transversal of $M$ in the proof of Theorem \ref{Construction--complete mapping}.
Furthermore, when $s \neq s' \in \mathbb{Z}_v$, the two transversals $T_s$ and $T_{s'}$ lie on  distinct pan-diagonals in the main diagonal direction.
Thus, all pan-diagonals of $M$ in the main diagonal direction
form transversals.

Similarly, for any $s \in \mathbb{Z}_v$, define
\[
T'_s = \{\ket{\phi_{i,j}} ~|~ i+j \equiv s \pmod{v}\}. 
\]
Since $\mu -\mathrm{id}$ is a permutation of $\mathbb{Z}_v$, we have
\[
 \langle \phi_{i,v-1-i+s} \mid \phi_{i',v-1-i'+s} \rangle
=\frac{1}{v}\sum\limits_{t=0}^{v-1}\omega^{t(v-1-i'+s)+i'\mu (t)+(-t(v-1-i+s)-i\mu (t))}
=\frac{1}{v}\sum\limits_{t=0}^{v-1}\omega^{-(i'-i)(\mu (t)-t)} =0,
\]
so each $T'_s$ is also a transversal of $M$.
Furthermore, all pan-diagonals of $M$ in the anti-diagonal direction
form transversals,  i.e., $M$ is an MCPQLS$(v)$.

\begin{corollary} \label{MCPDMQLS}
       If there exists a strong complete mapping of \(\mathbb{Z}_v\) satisfying the condition $(*)$, then there  exists an MCPQLS$(v)$.
\end{corollary}

For convenience, in the concrete constructions that follows, for the condition $(*)$ that the complete mapping $\mu$ needs to satisfy, we further restrict $s$ and $t$ to $t = 0$ and $s = 1$. For any $k \in \mathbb{Z}_v$, we write
\begin{align} \label{s1}
    f(k) = k\mu(1) - \mu(k) - (k-1)\mu(0).
\end{align}
The condition $(*)$  then requires that there exists some $k_0 \in \mathbb{Z}_v \setminus \{0, 1\}$ such that $\gcd(f(k_0), v) = 1$.
When $v$ is prime, the condition further reduces to the existence of some $k_0 \in \mathbb{Z}_v \setminus \{0, 1\}$ such that $f(k_0) \not\equiv 0 \pmod{v}$.
In addition, for any $a, b \in \mathbb{R}$ with $a < b$, let $[a,b]$ denote the set of all integers in the interval $[a,b]$, and let $[a,b]_o$ and $[a,b]_e$ denote the sets of odd and even integers in $[a,b]$, respectively.

\begin{lemma} \label{MCIDQLS-odd}
       For any odd integer $v \geq 7$, there exists an idempotent  MCQLS$(v)$. 
 \end{lemma}
\begin{proof}
By Theorem \ref{Construction--complete mapping}, we only need to construct a complete mapping of \(\mathbb{Z}_v\) satisfying the condition $(*)$. We consider two cases.

\paragraph{Case 1.} $v \equiv 1 \pmod{4}$ and $v \geq 7$. Define a mapping $\mu: \mathbb{Z}_v \rightarrow \mathbb{Z}_v$ as follows:
  \begin{align*}
                \mu(k) = \begin{cases}
                    k, & k \in [0, \frac{v-5}{2}],\\
                    k+1, & k =\frac{v-3}{2},\\
           k+2, & k \in [\frac{v-1}{2}, v-3]_e,\\
                    k-2, & k \in [\frac{v+1}{2}, v-2]_o,\\
                     k-1, & k=v-1.\\
                \end{cases}
            \end{align*}

First, we show that $\mu$ is a permutation on $\mathbb{Z}_v$. Clearly, $\mu$ is well-defined. Since $\mathbb{Z}_v$ is finite, it suffices to show that $\mu$ is injective from $\mathbb{Z}_v$ to $\mathbb{Z}_v$.
Write $K_1 =  [0, \frac{v-5}{2}],  K_2 = [\frac{v-1}{2}, v-3]_e, K_3 = [\frac{v+1}{2},  v-2]_o$. Choose any $k_1, k_2\in \mathbb{Z}_v$, $k_1 \neq k_2$. If $k_1, k_2\in K_i$, $i\in\{1,2,3\}$, then $\mu (k_1)=k_1 + a_i\neq k_2+a_i=\mu (k_2)$, where $a_1=0,a_2=2,a_3=-2$. 
 If $k_1\in K_1$, $k_2\geq \frac{v-3}{2}$, then $\mu (k_1)=k_1\leq \frac{v-5}{2}<\frac{v-3}{2}\leq \mu (k_2)$.  If $k_1 = \frac{v-3}{2}$ and $k_2 \geq \frac{v-1}{2}$, then when $k_2 = \frac{v+1}{2}$, we have $\mu(k_2) = \frac{v-3}{2} \neq \mu(k_1)$; while when $k_2 \neq \frac{v+1}{2}$, $\mu(k_2) \geq \frac{v+1}{2} > \mu(k_1)$.
 If $k_1\in K_2, k_2\in K_3$, then $\mu(k_1)$ and $\mu(k_2)$ have opposite parity, hence they are not equal. If $k_2 = v-1$, note that $\mu(k_2) = v-2$ is odd. If $k_1 \in K_2$, then $\mu(k_1) = k_1 + 2$ is even, so $\mu(k_1) \neq \mu(k_2)$. If $k_1 \in K_3$, then $\mu(k_1) \leq v-4$, and thus $\mu(k_1) \neq \mu(k_2)$.

Next we show $\mu+\mathrm{id}$ is a permutation on $\mathbb{Z}_v$, where
\begin{align*}
               (\mu+\mathrm{id})(k) = \begin{cases}
                    2k, & k \in [0, \frac{v-5}{2}],\\
                   2 k+1, &k =\frac{v-3}{2},\\
                    2k+2, & k \in [\frac{v-1}{2}, v-3]_e,\\
                    2k-2, & k \in [\frac{v+1}{2}, v-2]_o,\\ 
                    2 k-1, & k =v-1.
                \end{cases}
            \end{align*}
Similarly, it suffices to show that $\mu + \mathrm{id}$ is injective from $\mathbb{Z}_v$ to $\mathbb{Z}_v$.
Choose any $k_1,k_2\in \mathbb{Z}_v$, $k_1 \neq k_2$. If $k_1, k_2\in K_i$, $i\in\{1,2,3\}$, then $(\mu+\mathrm{id})  (k_1)=2(k_1+b_i) \neq 2(k_2+b_i)=(\mu+\mathrm{id})  (k_2)$ with $b_1=0,b_2=1,b_3=-1$ as $(2,v)=1$. 
If $k_1\in K_1$, $k_2\geq \frac{v-3}{2}$, then when $k_2=\frac{v-3}{2}$ or $v-1$, we have $(\mu+\mathrm{id})  (k_1)\leq v-5\leq v-3\leq(\mu+\mathrm{id})  (k_2)$; while when $k_2\in K_2\cup K_3$, $(\mu+\mathrm{id})  (k_1)=2k_1\equiv(\mu+\mathrm{id})  (k_2)=2k_2\pm 2 \pmod{v}$ would lead
$k_1\equiv k_2\pm 1  \pmod{v}$, which is impossible. 
If $k_1=\frac{v-3}{2}$ and $k_2=v-1$, then clearly $(\pi+\mathrm{id})(k_1)=v-2\neq(\pi+\mathrm{id})(k_2)=v-3$; moreover, if $k_2\in K_2\cup K_3$, to have $(\pi+\mathrm{id})(k_1)=v-2 \equiv 2k_2\pm 2 = (\pi+\mathrm{id})(k_2) \pmod{v}$, then it must be that when $k_2\in K_2$, $k_2+1\equiv -1\pmod{v}$, and when $k_2\in K_3$, $k_2\equiv 0\pmod{v}$, which is also impossible.
Similarly, if $k_1 \in K_2 \cup K_3$ and $k_2 = v-1$, then the congruence $(\mu + \mathrm{id})(k_1) \equiv (\mu + \mathrm{id})(k_2) \pmod{v}$ would require either $k_1 + 1 \equiv \frac{v-3}{2} \pmod{v}$ (when $k_1 \in K_2$) or $k_1 - 1 \equiv \frac{v-3}{2} \pmod{v}$ (when $k_1 \in K_3$), neither of which is possible.
 And if $k_1\in K_2$, $k_2\in  K_3$, then for $(\mu+\mathrm{id})  (k_1)=2k_1+2\equiv (\mu+\mathrm{id}) (k_2)=2k_2-2 \pmod{v}$ to hold, we must have $k_1+1=k_2-1$ in $Z_v$, which is also impossible. 

Finally, take
$k_0 = v-1$. Since
\begin{align*}
    f(k_0) = k_0\mu(1) - \mu(k_0) - (k_0-1)\mu(0) = v-1 - (v-2) - 0 = 1,
\end{align*}
and $\gcd(v, f(k_0)) = 1$, it follows that $\mu$ satisfies condition $(*)$.
\paragraph{Case 2.} $v \equiv 3 \pmod{4}$ and $v \geq 7$. Define a mapping $\pi: \mathbb{Z}_v \rightarrow \mathbb{Z}_v$ as follows:
              \begin{align*}
                \mu(k) = \begin{cases}
                    k, & k \in [0, \frac{v-5}{2}],\\
                    k+1, & k \in \{\frac{v-3}{2}, v-2\},\\
                    k+2, & k \in [\frac{v-1}{2}, v-4]_o,\\
                    k-2, & k \in [\frac{v+1}{2}, v-3]_e \cup \{v-1\}.
                \end{cases}
            \end{align*}
 Similar to the previous case, one can verify that both $\mu$ and $\mu + \mathrm{id}$ are injective maps from $\mathbb{Z}_v$ to $\mathbb{Z}_v$.
take
$k_0 = v-1$. Since
\begin{align*}
    f(k_0) = k_0\mu(1) - \mu(k_0) - (k_0-1)\mu(0) = v-1 - (v-3) - 0 = 2,
\end{align*}
and $\gcd(f(k_0), v) = 1$, it follows that $\mu$ satisfies condition $(*)$.\qed
\end{proof}

By the discussion before  Corllary \ref{MCPDMQLS}, we also have the following conclusion.
\begin{lemma} \label{MCPDQLS-odd}
       For any odd integer $v \geq 7$, there exists an $\mathrm{MCQLS}(v)$ having $n$ disjoint transversals.
 \end{lemma}

For a prime power $q$, Let $\mathbb{F}_q$ denote  the finite field containing $q$ elements. A polynomial $g\in \mathbb{F}_q[x]$ is a permutation polynomial of
if the function $g:c\mapsto f(c)$ from $\mathbb{F}_q$ to  itself induces a permutation. We will attempt to obtain strong complete mappings satisfying the  condition $(*)$ on cyclic groups by starting from permutation polynomials.    The following two conclusions are needed, one of which involves cyclotomic classes and multiplicative characters in finite fields. Here, fix a prime power $q \equiv 1 \pmod{e}$ and a primitive element $\omega \in \mathbb{F}_q$. Let $C_{0}^e$ denote the multiplicative subgroup
$\{\omega^{ie} \mid 0 \leq i < \frac{q-1}{e}\}$ of $\mathbb{F}_q$, and let $C_{j}^e = \omega^j C_{0}^e$ denote the coset of $C_{0}^e$
in $\mathbb{F}_q^* = \mathbb{F}_q \setminus \{0\}$.

\begin{lemma}\label{permutation polynomial-2}\cite{Lidl-Rudolf-Harald Niederreiter94}
For any prime power $q$,  $g(x)=ax^{k}$ is a  permutation polynomial of $GF(q)$ if and only if $a\neq 0$ and $\gcd(k,q-1)=1$.

\end{lemma}

\begin{lemma}\label{permutation polynomial}\cite{Lidl-Rudolf-Harald Niederreiter94}
For odd prime power $q$,  $g(x)=x^{\frac{q+1}{2}}+ax$ is a  permutation polynomial of $\mathbb{F}_q$ if and only if $a^2-1$ is a nonzero square in $\mathbb{F}_q$.

\end{lemma}

\begin{lemma}\label{Chang-Ji}
\cite{Chang-Ji 2004}
Let $p$ be a prime $ \equiv 1\ \pmod{e}$  with  $N_s(e,p) \triangleq p - [\sum\limits_{i=0}^{s-2} \binom{s}{i}(s-1-i)(e-1)^{s-i}]\sqrt{p} - se^{s-1} > 0$.
       Then, for any given $s$-tuple $(j_1, j_2, \dots, j_s) \in \{0,1,\dots,e-1\}^s$ and any given $s$-tuple $(c_1, c_2, \dots, c_s)$ of pairwise distinct elements of $\mathbb{F}_p$, there exists an element $x \in \mathbb{F}_p$ such that $x + c_i \in C_{j_i}^{e}$ for each $i$.
\end{lemma}

\begin{lemma} \label{MCDQLS-odd prime} 
For any odd prime  $p\geq 13$, there exists  an MCDQLS$(p)$.
\end{lemma}
\begin{proof}
By Theorem \ref{Construction--complete mapping}, we only need to construct a  strong complete mapping of \(\mathbb{Z}_p\) satisfying the  condition $(*)$.

Taking $e=2$, $s=5$,  $(c_1,c_2,c_3,c_4,c_5)=(-2,-1,0,1,2)$  and $(j_1,j_2,j_3,j_4,$ $j_5)=(0,0,0,0,0)$ in lemma \ref{Chang-Ji}, a simple calculation shows that when $p \geq 2559$, we have $N_s(e, p) > 0$. Then
 there exists an element $a\in \mathbb{Z}_p$ such that $a^2-1, (a+1)^2-1$ and $(a-1)^2-1$ are all nonzero square in $\mathbb{Z}_p$.
Furthermore, by Lemma \ref{permutation polynomial},  $g(x)=x^{\frac{p+1}{2}}+ax$, $h(x)=g(x)-x$
and $l(x)=g(x)+x$ are all  permutation polynomials over $\mathbb{Z}_p$, i.e., $g$ is a strong complete mapping of $\mathbb{Z}_p$. For any prime $p\in(11,2467)$ and $p\neq 13,17,29,37$,  the element $a$ satisfying the condition that $a^2-1, (a+1)^2-1$ and $(a-1)^2-1$ are all nonzero square in $\mathbb{Z}_p$ can be  found  easily  with the aid of computer. The following table lists our search results for the first $10$
values of $p$.

$$\begin{array}{c|ccccccccccc}
\hline
p&19& 23 & 31 & 41 & 43 & 47 & 53 & 59 & 61 &67\\
\hline
\omega &2 & 5 & 3 & 6 & 3 & 5 & 2 & 2 & 2 & 2 \\

a&9 & 6 & 15 & 13 & 5 & 13 & 13 & 7 & 14 &5 \\
\hline
\end{array}$$ 

Clearly, for the above permutation polynomial  $g(x)=x^{\frac{p+1}{2}}+ax$, we have that $g(0)=0$ and $g(\omega)=\omega ^{\frac{p+1}{2}}+a\omega=\omega(a-\omega) \neq \omega g(1)=\omega(a+1)$, where $\omega$ is a primitive element.
 Since $f(\omega) = \omega g(1) - g(\omega) - (\omega-1)g(0) \not\equiv 0 \pmod{p}$, $g$ satisfies the condition $(*)$ as a permutation on $\mathbb{Z}_p$.
 
For $p\in \{13,17,29,37\}$, take $t_p=6$ if $p\in \{13,17,29\}$ and take  $t_p=2$ if $p=37$, then $g(x)=t_px^{\frac{p+1}{2}}$  is a permutation polynomial of $\mathbb{Z}_p$ as
$p \equiv 1\ \pmod{4}$ and $(\frac{p+1}{2},p-1)=1$. Furthermore, we can verify that $t_p^{-2} - 1 = t_p^{-2}(1 - t_p^2) \in C_{0}^2$, and thus $g(x) \pm x = t_p(x^{\frac{p+1}{2}} \pm t_p^{-1}x)$ is also a permutation polynomial over $\mathbb{Z}_p$. Clearly, we also have $f(\omega _p) = \omega g(1) - g(\omega _p) - (\omega_ p-1)g(0) \not\equiv 0 \pmod{p}$, where $\omega _p =2$ if $p\in \{13,29,37\}$, and $\omega _p =3$ if $p=17$. Thus $g$  satisfies the condition $(*)$ as a permutation on $\mathbb{Z}_p$.\qed

\end{proof}

\begin{lemma} \label{MCDQLS-odd} 
For any odd composite integer \(v\) with \(v \geq 25\) and \(3 \nmid v\), there exists an  MCDLS$(v)$.
\end{lemma}

\begin{proof}
 By Theorem \ref{Construction--complete mapping}, we only need to construct a strong complete mapping of \(\mathbb{Z}_v\) satisfying the  condition $(*)$.

Suppose $v = pt$, where $p$ is the least prime factor of $v$ and $t$ is odd.
Let $A=\{0,p,2p,\dots,(t-1)p\}$, and $B=\{0,1,2,\dots,p-1\}$. Then for $(\mathbb{Z}_v,+)$, $B$ is a coset representative system of $A$. Divide $B$ into three parts: $B_{0}=\{0\}$, $B_{1}=\{1,2,\dots,\frac{p-1}{2}\}$ and $B_{2}=\{\frac{p+1}{2},\frac{p+3}{2},\dots,p-1\}$.
Now we  construct a mapping $\mu: \mathbb{Z}_v \rightarrow \mathbb{Z}_v$ as follows:
\begin{align*} \label{2}
            \mu(x) = \begin{cases}
        2x, & \text{if } x \in (A \setminus \{(t-1)p\})+(B_{0}\cup B_{1}), \\
        2x-p, & \text{if } x \in (A \setminus \{(t-1)p\})+B_{2}, \\
        2x, & \text{if }  x=(t-1)p, \\
        2x-\frac{3p-1}{2}+3s+1, &  \text{if } x=tp-1-2s,s \in \{0,1,\dots,\frac{p-1}{2}-1\}, \\
        2x-p+3s+3, &  \text{if } x=tp-1-2s-1, s \in \{0,1,\dots,\frac{p-1}{2}-1\}.
        \end{cases}
        \end{align*}
Clearly, $\mu$ is well-defined. Take $k_0 = v-1$ and we have
\begin{align*}
            f(k_0) &= (v-1)\mu(1) - \mu(v-1) - (v-1-1)\mu(0)\\
            &=2(pt-1) - [2(pt-1) - 3 \cdot \frac{p-1}{2}] - 0 = 3 \cdot \frac{p-1}{2}.
        \end{align*}      
        Since $p$ is the smallest prime factor of $v$, we must have $\gcd(t, \frac{p-1}{2}) = 1$ and $\gcd(v, 3) = 1$, then
        $$\gcd(v, 3 \cdot \frac{p-1}{2}) = \gcd(pt, \frac{p-1}{2}) = \gcd(t, \frac{p-1}{2}) = 1,$$
        hence $\mu$ is a mapping of \(\mathbb{Z}_v\) satisfying the  condition $(*)$. 
Note that $2 \nmid v$ and $3 \nmid v$, we have $2\mathbb{Z}_{v} = 3\mathbb{Z}_{v} = \mathbb{Z}_{v} = A + B$. Before considering the images of $\mu - \mathrm{id}$, $\mu$, and $\mu + \mathrm{id}$ on $\mathbb{Z}_v$, note that
\begin{align*}
       &\{ 2x\colon x\in \frac{3(t-1)}{2}p+B_{2}\} \cup \{{2x\colon x\in (t-1)p+B_{1}}\} \\
       =& \{ 2x-p\colon x\in (t-1)p+B_{2}\}\cup \{{2x\colon x\in (t-1)p+B_{1}}\}\\
       = &\{(t-2)p+s\colon s=1,3,\dots ,p-2\}\cup \{(t-2)p+s\colon s=2,4,\dots ,p-1\}\\
       = &(t-2)p+(B_1\cup B_2)
       \end{align*}
      and 
      \begin{align*}
        &\{ 3x\colon x\in \frac{4(t-1)}{3}p+B_{2}\}\cup \{{3x\colon x\in (t-1)p+B_{1}}\}\\
        = &\{ 3x-p\colon x\in (t-1)p+B_{2}\}\cup \{{3x\colon x\in (t-1)p+B_{1}}\}\\
        = &\{(t-4)p+3s\colon s=\frac{p+1}{2},\frac{p+3}{2},\dots, p-1 \}\cup \{(t-3)p+3s\colon s\in s=1,2,\dots,\frac{p-1}{2}\}.
      \end{align*}
Then we have the following calculation results.

      $$\begin{array}{lll}
      &\{(\mu - \mathrm{id})(x)\colon x \in \mathbb{Z}_{v}\} \\
      =&\{x\colon x \in(A \setminus \{(t-1)p)\})+(B_{0}\cup B_{1})\}\cup \{x-p\colon x \in(A \setminus \{(t-1)p\})+B_{2}\}\cup \{(t-1)p\} \cup\\
      & \{x-\frac{3p-1}{2}+3s+1\colon x=tp-1-2s, s=0,1,\dots,\frac{p-1}{2}-1\} \cup\\
      & \{x-p+3s+3\colon x=tp-1-2s-1, s=0,1,\dots,\frac{p-1}{2}-1\}\\
      =&((A \setminus \{(t-1)p\})+(B_{0}\cup B_{1})) \cup ((A \setminus \{(t-2)p\})+B_{2}) \cup \{(t-1)p\} \cup \\
      &  \{(t-2)p+\frac{p+1}{2}+s\colon s=0,1,\dots,\frac{p-1}{2}-1\} \cup \{(t-1)p+1+s\colon s=0, 1, \dots, \frac{p-1}{2}-1\}\\
      =&((A \setminus \{(t-1)p\})+(B_{0}\cup B_{1})) \cup ((A \setminus \{(t-2)\})+B_{2}) \cup \{(t-1)p\} \cup ((t-2)p+B_{2}) \cup\\
      &  ((t-1)p+B_{1}\})\\
      =&A+B\\
      =&\mathbb{Z}_{v}.
      \end{array}$$

      $$\begin{array}{lll}
       &\{\mu(x)\colon x \in \mathbb{Z}_{v}\}\\
       =&\{2x\colon x \in (A \setminus \{(t-1)p\})+(B_{0}\cup B_{1})\} \cup \{2x-p\colon x \in (A \setminus \{(t-1)p\})+B_{2}\}\cup \\
       & \{2x\colon x=(t-1)p\} \cup\{2x-\frac{3p-1}{2}+3s+1\colon x= tp-1-2s,s=0,1,\dots,\frac{p-1}{2}-1\} \cup\\
       &\{2x-p+3s+3\colon x=tp-1-2s-1, s=0,1,\dots,\frac{p-1}{2}-1\}\\
       = &2((A \setminus \{(t-1)p\})+(B_{0}\cup B_{1})) \cup \{2(x+\frac{t-1}{2}p)\colon x \in (A\setminus\{(t-1)p\})+B_{2}\} \cup\\
       & \{2(t-1)p\} \cup \{(t-2)p+\frac{p-1}{2}-s\colon s=0, 1, \dots, \frac{p-1}{2}-1\} \cup \\
       & \{(t-1)p-s-1\colon  s=0, 1,\dots, \frac{p-1}{2}-1\}\\
       =&2\{(A \setminus \{(t-2)p\})+(B_{0}\cup B_{1})\} \cup 2\{(A \setminus \{\frac{3(t-1)}{2}p\})+B_{2}\} \cup \{2(t-1)p\} \cup\\
        & \{(t-2)p+\frac{p-1}{2}-s\colon s =0,1,\dots,\frac{p-1}{2}-1\}\\
       &\cup \{(t-2)p+p-s-1\colon s=0, 1, \dots, \frac{p-1}{2}-1\}\\
      =&2\{(A \setminus \{(t-1)p\})+(B_{0}\cup B_{1})\} \cup 2\{(A \setminus\{\frac{3(t-1)}{2}p\})+B_{2}\}\cup \{2(t-1)p\} \cup\\
       & ((t-2)p+B_{1}) \cup( (t-2)p+B_{2})\\
      =&2\{(A \setminus \{(t-1)p\})+(B_{0}\cup B_{1})\} \cup 2\{(A \setminus\{\frac{3(t-1)}{2}p\})+B_{2}\}\cup \{2(t-1)p\} \cup\\
      &( (t-2)p+(B_{1} \cup B_{2}))\\
       =&2(A+B)\\
      =&2\mathbb{Z}_{v}\\
      =&\mathbb{Z}_{v}.
      \end{array}$$

      $$\begin{array}{lll}
       &\{(\mu +2\mathrm{id})(x)\colon x \in \mathbb{Z}_{v}\}\\
       =&\{3x\colon  x \in (A \setminus \{(t-1)p\})+(B_{0}\cup B_{1})\} \cup \{3x-p\colon x \in (A \setminus \{(t-1)p\})+B_{2}\}\\
       &\cup \{3(t-1)p\} \cup \{3x-\frac{3p-1}{2}+3s+1\colon x= tp-1-2s,s=0,1,\dots,\frac{p-1}{2}-1\}\\
        &\cup \{3x-p+3s+3\colon x=tp-1-2s-1, s=0,1,\dots,\frac{p-1}{2}-1\}\\
        =&3((A \setminus \{(t-1)p\})+(B_{0}\cup B_{1})) \cup \{3(x+\frac{t-1}{3}p)\colon x \in (A\setminus\{(t-1)p\})+B_{2}\}\\
        &\cup \{3(t-1)p\} \cup \{(t-2)p+\frac{p+1}{2}-3s-2\colon  s=0,1,\dots,\frac{p-1}{2}-1\}\}\\
         &\cup \{(t-1)p-3s-3\colon s=0,1,\dots,\frac{p-1}{2}-1\}\\
        =&3((A \setminus \{(t-1)p\})+(B_{0}\cup B_{1})) \cup 3((A \setminus\{\frac{4(t-1)}{3}p\})+B_{2})\\
        &\cup \{3(t-1)p\} \cup \{(t-3)p+3(\frac{p-1}{2}-s)\colon s=0,1,\dots,\frac{p-1}{2}-1\}\}\\
        &\cup \{(t-4)p+3(p-1-s)\colon s=0,1,\dots,\frac{p-1}{2}-1\}\\
        =&3((A \setminus \{(t-1)p\})+(B_{0}\cup B_{1})) \cup 3((A \setminus\{\frac{4(t-1)}{3}p\})+B_{2})\\
        &\cup \{3(t-1)p\} \cup \{(t-3)p+3s\colon  s=1,2,\dots,\frac{p-1}{2}\}\\
        &\cup \{(t-4)p+3s\colon s=\frac{p+1}{2},\frac{p+3}{2},\dots, p-1 \}\\
        =&3\mathbb{Z}_v\\
        =&\mathbb{Z}_v.
        \end{array}$$

Therefore, $\mu$ is a required strong complete mapping of \(\mathbb{Z}_v\) satisfying the  condition $(*)$. \qed
       
\end{proof}

By Corllary \ref{MCPDMQLS} and the proof of Lemmas \ref{MCDQLS-odd prime}, \ref{MCDQLS-odd}, we also have the following result.

\begin{theorem} \label{MCPDQLS-odd-1} 
    There exists an $\mathrm{MCPQLS}(v)$ if  $\gcd(n,6)=1$ and  $v\not \in \{5,7,11\}$. 

\end{theorem}

\section{Recursive Constructions}
In this section, we introduce a singular direct product construction for MCDQLSs, which can be seen as a generalization of Theorem~3.4 in Ref.~\cite{Cao-July}; both this method and the theorem generalize the singular direct product construction for classical Latin squares.

\subsection{Singular direct product constructions}
We introduce the following definitions and notations, which will be used throughout the rest of this paper.

 An incomplete quantum Latin square, denoted by IQLS\( (m+h, h) \), is an \( (m+h) \times (m+h) \) square  whose entry  in the \( h \times h \) subblock at its lower right corner   is empty, satisfying the following conditions:
\begin{enumerate}
    \item[(1)] Each nonempty entry is a unit column vector from \( (m+h) \)-dimensional Hilbert space \( \mathcal{H}_{m+h} \);

    \item[(2)]  For \( i \in \{0, 1, \dots, m-1 \} \), the entries in the \( i \)-th row (column) form an orthonormal basis of \( \mathcal{H}_{m+h} \);

    \item[(3)]  For \( j \in \{m, m+1, \dots, m+h-1 \} \), all the nonempty entries in the \( j \)-th row (column) form an orthonormal basis of the space generating by  the computational vectors \( \ket{0}, \ket{1}, \dots, \ket{m-1} \in \mathcal{H}_{m+h}\).
\end{enumerate}
Here, if $h = 0$, then $\operatorname{IQLS}(m+h, h)$ reduces to a $\operatorname{QLS}(m)$.
Furthermore, we define the cardinality \( c \) of an IQLS(\( m+h, h \)) to be the number of distinct vectors (up to a global phase) in the array. Clearly, the cardinality satisfies \( n \leq c \leq (m+h)^2 - h^2 \). If \( c = (m+h)^2 - h^2 \), then the array is an IQLS\((m+h, h)\) of maximum cardinality, denoted by \(\mathrm{MCIQLS}(m+h, h)\).

The $\operatorname{IQLS}(m+h, h)$ is the quantum generalization of the classical incomplete Latin square $\operatorname{ILS}(m+h, h)$. The latter is an $(n+h) \times (n+h)$ array with the $h \times h$ subblock at its lower right corner left empty, such that every symbol occurs at most once in each row and each column, and no symbol from an $h$-subset of the symbol set occurs in the last $h$ rows nor in the last $h$ columns.
For two MCQLS$(m+h,h)$s $A$ and $B$, if every entry in $A$ is distinct from every entry in $B$, then we say that $A$ and $B$ are disjoint. In addition, we also need to consider the disjointness between different MCIQLSs.

For any vector $\ket{c}=(c_0, c_1, \dots, c_{m+h-1})^{\top}\in \mathcal{H}_{m+h}$, define 
$\ket{c_{[0,m-1]}} = (c_0, c_1, \dots, c_{m-1})^{\top}$ as the projection onto the first $m$ components of $\ket{c}$, and 
$\ket{c_{[m,m+h-1]}} = (c_m, c_{m+1}, \dots, c_{m+h-1})^{\top}$ as the projection onto the last $h$ components.
Then given a positive integer $m$ and nonnegative integers $m_p, m_q$, for two unit vectors 
$\ket{a} \in \mathcal{H}_{m+m_p}$ and $\ket{b} \in \mathcal{H}_{m+m_q}$ , if the unit vectors
$$\begin{pmatrix}
              \ket{a_{[0,m-1]}} \\ \ket{a_{[m,m+m_p-1]}}\\ \mathbf{0}_{m_q}
          \end{pmatrix}, \  \begin{pmatrix}
           \ket{b_{[0,m-1]}} \\ \mathbf{0}_{m_p}    \\ \ket{b_{[m,m+m_q-1]}}
          \end{pmatrix}$$ are not distinct  (up to a global phase),
        then $\ket{a}$ and $\ket{b}$ are said to be $m$-distinct.
 Under this definition, if $m_p = m_q=0$, then the distinctness of $\ket{a}$ and $\ket{b}$ is equivalent to their $m$-distinctness. If $m_p = 0$ and $\ket{b_{[m,m+m_q-1]}}=\mathbf{0}_{m_q}$ , then $\ket{a}$ and $\ket{b}$ being $m$-distinct implies that the projection of $\ket{b}$ onto its first $m$ components is distinct from $\ket{a}$. 
 Moreover, when $m_p = m_q$, the distinctness of $\ket{a}$ and $\ket{b}$ clearly implies that $\ket{a}$ and $\ket{b}$ are $m$-distinct. Two $\mathrm{IQLS}(m+h, h)$'s are said to be disjoint if every nonzero entry of one is distinct from every nonzero entry of the other. And one can also define the notion of $m$-disjointness between a quantum Latin square and an incomplete quantum Latin square, as well as between different incomplete quantum Latin squares.
For example, an IQLS$(m+h, h)$ and a QLS$(m)$ are said to be  $m$-disjoint if every nonempty entry in the IQLS$(m+m_p, m_p)$ is $m$-distinct from every entry in the QLS$(m)$. 
Similarly, one can also speak of an IQLS$(m + m_p,m_p)$ and an IQLS$(m + m_q,m_q)$ being $m$-disjoint.

Let $m$ be $n$ positive integer, and let $w, m_1, \dots, m_p, \dots, m_u$ be nonnegative  integers satisfying $w = \displaystyle\sum_{p=1}^{u}m_p$.  
Consider the following quantum state vectors:
\[
\ket{a} =
\begin{pmatrix}
a_0 \\ a_1 \\ \vdots \\ a_{n-1}
\end{pmatrix}
\in \mathcal{H}_n,
\quad
\ket{b} =
\begin{pmatrix}
b_0 \\ b_1 \\ \vdots \\ b_{m-1}
\end{pmatrix}
\in \mathcal{H}_m,
\quad
\ket{c} =
\begin{pmatrix}
c_0 \\ c_1 \\ \vdots \\ c_{m-1}\\ c_{m}\\ \vdots \\ c_{m+m_p-1}
\end{pmatrix}
\in \mathcal{H}_{m+m_p}\ \ (1\leq  p \leq u).
\]
   Define the extended tensor product operation $\otimes_{+}$ as
    \[
    \ket{a} \otimes_{+} \ket{b} =
    \begin{pmatrix}
    \ket{a} \otimes \ket{b} \\ \mathbf{0}_w
    \end{pmatrix}
    \in \mathcal{H}_{mn+w}.
    \]
 And define the operation $\otimes_{\widehat{p}}$ as
    \[
    \ket{a} \otimes_{\widehat{p}} \ket{c} =
    \begin{pmatrix}
    \ket{a} \otimes \ket{c_{[0,m-1]}} \\ \mathbf{0}_{m_{1}}\\ \vdots \\ \mathbf{0}_{m_{p-1}} \\ \ket{c_{[m,m+m_p-1]}}\\ \mathbf{0}_{m_{p+1}} \\ \vdots \\ \mathbf{0}_{m_{u}}
    \end{pmatrix}
    \in \mathcal{H}_{mn+w}.
    \]

Since $\ket{a}$, $\ket{b}$ and $\ket{c}$ are all unit vectors, we have
\[
\|\ket{a} \otimes_{+} \ket{b}\| = \|\ket{a}\| \cdot \|\ket{b}\| = 1,
\]
and
\[
\|\ket{a} \otimes_{\widehat{p}} \ket{c}\| = \sqrt{\|\ket{a}\|^2 \cdot \|\ket{c_{[0,m-1]}}\|^2 + \|\ket{c_{[m,m+m_p-1]}}\|^2} = \|\ket{c}\| = 1.
\]
Thus, $\ket{a} \otimes_{+} \ket{b}$ and $\ket{a} \otimes_{\widehat{p}} \ket{c}$ are both unit vectors.

To facilitate the construction methods described below, 
consider  an MCIQLS$(m+m_p,m_p)$ $C_p = (\ket{c_{p,s,t}})$,  $p=1,2,\dots,u$,  and an  MCQLS$(w)$ $D = (\ket{d_{e,f}})$.  
If, for every nonempty entry $\ket{c_{p,s,t}}$ in $C_p$,  and every  entry $ \ket{d_{e,f}}$ in $D$, 
$$\begin{pmatrix}
        \ket{c_{p,s,t,[0,m-1]}} \\ \mathbf{0}_{m_{1 }}\\ \vdots \\ \ket{c_{p,s,t,[m,m+m_p-]}}\\ \vdots \\ \mathbf{0}_{m_{u}}
    \end{pmatrix}\neq\begin{pmatrix}
        \mathbf{0}_{m} \\ \ket{d_{e,f}}
    \end{pmatrix},$$ 
then $D$ and $C_p$ are said to be $(m,p;m_1, \dots, m_p, \dots, m_u)$-type disjoint. 

Clearly,  when $m_1 = \dots = m_p = \dots = m_u = 1$, the condition of $(m,p;m_1, \dots, m_p, \dots, m_u)$-type disjointness for an MCIQLS$(m+m_p,m_p)$ and an MCQLS$(w)$ is automatically satisfied.And if the $m$-dimensional projections onto the first $m$ components of those entries in the $\mathrm{MCIQLS}(m+m_p,m_p)$ ($p=1,2,\dots,u$) are all nonzero, then the condition of $(m,p;m_1, \dots, m_p, \dots, m_u)$-type disjointness for an $\mathrm{MCIQLS}(m+m_p,m_p)$ and an $\mathrm{MCQLS}(w)$ is also automatically satisfied.

\begin{theorem} \label{MCDQLS-SDPC}
Let $n, m, u, m_1, \ldots, m_p, \ldots, m_u, w$ be positive integers, where $n$ is even, $0 \le u \le n-2$, and $w = \displaystyle\sum_{p=1}^{u} m_p$.
Suppose there exists a DLS$(n)$ possessing $u+2$ pairwise disjoint transversals, including the main diagonal and the anti-diagonal.
Furthermore, if
\begin{enumerate}
    \item[$(1)$] there exist $n+2$ pairwise disjoint MCQLS$(m)$s, two of which are idempotent;
    \item[$(2)$]  there exist $u$ pairwise $m$-disjoint MCIQLSs, which are MCIQLS$(m+m_p, m_p)$, $1 \le p \le u$, respectively, and all of them are $m$-disjoint from the MCQLS$(m)$s in $(1)$;
    \item[$(3)$]  there exists an MCDQLS$(w)$ that is $(m,p;m_1, \ldots, m_p, \ldots, m_u)$-type disjoint from the MCIQLS$(m+m_p, m_p)$ in $(2)$,
\end{enumerate}
then there exists an MCDQLS$(mn+w)$.
\end{theorem}

\begin{proof}
Replacing each element $i\in \mathbb{Z}_n$ in a DLS$(n)$ with $\ket{i}\in \mathcal{H}_n$ yields a DQLS$(n)$.
  Suppose that the classical DQLS$(n)$ corresponding to an existing DLS$(n)$ is denoted by $A =
(\ket{a_{i,j}})$. When $u \neq 0$, the $u$ disjoint transversals other than the main and anti-diagonals are denoted by $T_1, T_2, \dots, T_u$, respectively. 
For $1 \leq p \leq u$, to make the notation clearer, an entry $\ket{a_{i,j}}$ in the transversal $T_p$ is also denoted by $\ket{a_{p,i,j}}$. Suppose that two existing idempotent MCQLS$(m)$s are $B_{-1} = (\ket{b_{-1,k,l}})$ and $B_{-2} = (\ket{b_{-2,k,l}})$, 
where we may assume that $B_{-2}$ has been subjected to a column reversal permutation so that the main diagonal transversal is transformed into the anti-diagonal transversal. Moreover, the remaining MCQLS$(m)$s are $B_0, B_1, \dots, B_{n-1}$, where $B_i = (\ket{b_{i,k,l}})$. Suppose the existing IQLS$(m+m_p, m_p)$ is $C_p = (\ket{c_{p,s,t}})$, $(1 \leq p \leq u)$, and  it is partitioned into blocks as follows:
 \[
        C_p = \begin{pmatrix}
            C_{p1} & C_{p3}\\
            C_{p2} & \
        \end{pmatrix},
        \]
  where $C_{p1}$ is an $m\times m$  subarray. Suppose the existing MCDQLS$(w)$ is $D = (\ket{d_{e,f}})$.

  We construct a square of order \( (mn + w) \), denoted as \( M \). First,
  the subarray of order $mn$  denoted by $M_0$ in the upper left corner of $M$ is to be partitioned into $n^2$ disjoint $m \times m$ blocks, where the $(i, j)$-block is occupied by the subarray $M_{i,j}$ $(i, j = 0, 1, \dots, n-1)$, which is constructed as follows:
\[
M_{i,j} =
\begin{cases}
\ket{a_{i,j}} \otimes_{+}  B_{i} & \text{if } \ket{a_{i,j}} \notin (\bigcup\limits_{k=1}^{u} T_k), j\neq i, n-1-i, \\
\ket{a_{i,j}} \otimes_{+}  B_{-1} & \text{if } j=i, \\
\ket{a_{i,j}} \otimes_{+}  B_{-2} & \text{if } j=n-1-i, \\
\ket{a_{p,i,j}} \otimes_{\widehat{p}} C_{p1} & \text{if } \ket{a_{p,i,j}} \in T_p, \ \ 1 \leq p \leq u.
\end{cases}
\]
Next, the \( w \times mn \)  subarray  in the lower-left corner of \( M \)  is partitioned into  $un$ blocks,  where its $(p-1, j)$-block $(p = 1, 2, \dots, u, j = 0, 1, \dots, n-1 )$ is occupied by an $m_p \times m$ subarray  
\[
M_{n+p-1,j} = \ket{a_{p,i,j}} \otimes_{\widehat{p}} C_{p2}.
\]  Note that here, since $\ket{a_{p,i,j}}$ is in  the transversal  $T_p$ of $A$, the pair $(p, j)$ uniquely determines an $i \in \mathbb{Z}_n$.

Similarly,     the \(mn \times w \)  subarray in the upper-right corner of \( M \)  is partitioned into  $un$ blocks,  where its $(i, p-1)$-block $ (i = 0, 1, \dots, n-1, p = 1, 2, \dots, u )$ is occupied by an $m \times m_p$ subarray 
\[
M_{i,n+p-1} = \ket{a_{p,i,j}} \otimes_{\widehat{p}} C_{p3}.
\]

 Finally, the  \( w\times w \) subarray in the lower-right corner of \( M \) is denoted as \( M_w \). The \( (e, f))  \) entry is given by
\[
\begin{pmatrix}
\mathbf{0}_{mn} \\
\ket{d_{e,f}}
\end{pmatrix},
\]
which corresponds to the \((mn + e)\)-th row and \((mn + f)\)-th column of \(M\). 

  In order to show that \( M \) is an MCQLS$(mn+w)$,  we first present the following two claims:
 \begin{enumerate}
    \item[] \textbf{Claim $1$.} \( M \) is a  QLS$(mn+w)$.
    \item[]  \textbf{Claim $2$.} The cardinality of \( M \) is the maximum.
\end{enumerate}
 The corresponding proofs are provided in
Appendix A. Then, we attempt to transform $M$ into a diagonal quantum Latin square by permuting its rows and columns.
we need to prove that the entries on the main diagonal and on the anti-diagonal of $M_0$
form an orthonormal basis of \(\operatorname{span}\{\ket{0}, \ket{1}, \dots, \ket{mn-1}\} \subset \mathcal{H}_{mn+w} \) respectively. Clearly, the entries in \( M \) are all unit vectors and the entries  on the main diagonal and the anti-diagonal of $M_0$ are all in \(\operatorname{span}\{\ket{0}, \ket{1}, \dots, \ket{mn-1}\}  \). So, it suffices to check the orthogonality of the entries on the main diagonal and on the 
anti-diagonal.  Consider any two  different entries
 \[
                \ket{\alpha} = \begin{pmatrix}
                    \ket{a_{i,i}} \otimes \ket{b_{-1,k,k}}\\ \mathbf{0}_w
                \end{pmatrix},
                \ket{\beta} = \begin{pmatrix}
                    \ket{a_{i',i'}} \otimes \ket{b_{-1,k',k'}}\\ \mathbf{0}_w
                \end{pmatrix},
                \] with $(i,k)\neq (i',k')$, where $\ket{a_{i,i}}$, $\ket{a_{i',i'}}$ are entries of $A$ and $\ket{b_{k,k}}$, $\ket{b_{k',k'}}$ are entries of $B_{-1}$. 
 If $i \neq i'$, then $\braket{a_{i,i} \mid a_{i',i'}} = 0$ since $A$ is a DQLS($m$). And if $i = i'$ and $k \neq k'$, then $\braket{b_{k,k} \mid b_{k',k'}} = 0$, since $B_{-1}$ is an idempotent MCQLS$(m)$. So we have
\[
(\alpha, \beta) = \braket{a_{i,i} \mid a_{i',i'}} \braket{b_{k,k} \mid b_{k',k'}} + 0 = 0.
\]i.e., $\ket{\alpha}$ and $\ket{\beta}$ are orthogonal.
Similarly, we can also verify that any two distinct entries on the anti-diagonal of $M_0$ are orthogonal.

Note that the entries on the main diagonal and on the anti-diagonal of \( M_w \) respectively form an orthonormal basis of \(\operatorname{span}\{\ket{mn}, \ket{mn+1}, \dots, \ket{w-1}\}\). By moving rows \(mn/2, mn/2+1, \dots, mn-1\) of \( M\) to the last \(mn/2\) rows and the corresponding columns to the last \(mn/2\) columns, we obtain an MCQLS$(mn+w)$  denoted as $N$. Since the entries on the main diagonal and on the anti-diagonal of \( N \) respectively form an orthonormal basis of \(\mathcal{H}_{mn+w}\), \( N \) is an  MCDQLS$(mn+w)$. \qed
\end{proof}

\begin{remark} \label{rmk2}
       If $u = n-2$ in condition $(1)$ of Theorem~\ref{MCDQLS-SDPC}, then the number of mutually disjoint sets $A$ in condition (2) can be changed from $n+2$ to $n$. Furthermore, if $u = 0$ in condition $(2)$, then both conditions $(2)$ and $(3)$ can be omitted. 
    \end{remark}

\begin{remark} \label{rmk1}
In order to apply Theorem~\ref{MCDQLS-SDPC} repeatedly, we will sometimes require in what follows that the MCQLS$(m)$, the MCIQLS$(m+m_p, m_p)$ for $1\le p\le u$, and the MCDQLS$(w)$ in conditions $(1)$, $(2)$, and $(3)$ all satisfy that the sum of all components of each of their elements is nonzero.
      Note that in the proof of Theorem~\ref{MCDQLS-SDPC}, $\ket{a_{i,j}}\in\{\ket{0}, \ket{1}, \dots, \ket{n-1}\}$. Under the new requirement, the sum of the components of each entry of the resulting MCDQLS$(mn+w)$ from Theorem~\ref{MCDQLS-SDPC} is nonzero.

    \end{remark}

   In fact the parameter $n$ in theorem \ref{MCDQLS-SDPC} can be odd. We give without proof the following simple construction.
    
\begin{lemma} \label{MCDQLS-SDPC2}
Let $n, m$ be positive integers. Suppose there exist a DLS$(n)$ and $n$ pairwise disjoint  MCDQLS$(m)$s, then there exists a  MCDQLS$(mn)$.
\end{lemma}

\begin{remark} \label{}
In Lemma~\ref{MCDQLS-SDPC2}, if the sum of the components of each entry of each MCDQLS$(m)$ is nonzero, then the sum of the components of each entry of the resulting MCDQLS$(mn)$ is also nonzero.
    \end{remark}

We also have the  singular direct product construction for the idempotent quantum Latin square with maximum cardinality.

\begin{theorem} \label{MCIDQLS-SDPC}
Let $n, m, u, m_1, \dots, m_p, \dots, m_u, w$ be positive integers, where $0 \le u \le n-1$, and $w = \displaystyle\sum_{p=1}^{u} m_p$.
Suppose there exists an idempotent LS$(n)$ possessing $u+1$ pairwise disjoint transversals, including the main diagonal.
Furthermore, if
\begin{enumerate}
    \item[$(1)$] there exist $n+1$ pairwise disjoint MCQLS$(m)$, one of which is idempotent;
    \item[$(2)$]  there exist $u$ pairwise $m$-disjoint MCIQLS, which are MCIQLS$(m+m_p, m_p)$, $1 \le p \le u$, respectively, and all of them are $m$-disjoint from the MCQLS$(m)$ in $(1)$;
    \item[$(3)$]  there exists an idempotent   MCQLS$(w)$ that is $(m,p;m_1, \dots, m_p, \dots, m_u)$-type disjoint from the MCIQLS$(m+m_p, m_p)$ in $(2)$,
\end{enumerate}
then there exists an idempotent  MCQLS$(mn+w)$.
\end{theorem}

\begin{remark}\label{sum--SDPC} For Theorem~\ref{MCIDQLS-SDPC}, we have a similar discussion as in Remark~\ref{rmk2} and Remark~\ref{rmk1}.
In Theorem~\ref{MCIDQLS-SDPC}, if we remove condition $(3)$, do not require the initial LS$(n)$ to be idempotent, and change condition $(1)$ to ``there exist $n$ pairwise disjoint $\mathrm{MCQLS}(m)$'', then we obtain an $\mathrm{MCIQLS}(mn+w,w)$.
As in the discussion in Remark~\ref{rmk1}, we can also require that the sum of the components of each entry of the resulting $\mathrm{MCIQLS}(mn+w,w)$ be nonzero.
  \end{remark}

In order to use Theorem~\ref{MCDQLS-SDPC} and  Theorem~\ref{MCIDQLS-SDPC} smoothly, we list below, in the form of a lemma, some processing methods that guarantee the $(m,p;m_1,\dots,m_p,\dots,m_u)$-type disjointness property  in the following text. Since the conclusion is obvious, we omit the proof.

\begin{lemma} \label{type disjoint}
Let $m, w, m_1, \dots, m_p, \dots, m_u$ be nonnegative integers satisfying $w = \displaystyle\sum_{p=1}^{u} m_p$ and $m_1 \leq \dots \leq m_p \leq \dots \leq m_u$. Suppose that an MCIQLS$(m+m_p;m_p)$ $C_p = (\ket{c_{p,s,t}})$ for $p=1,2,\dots,u$ and an MCQLS$(w)$ $D = (\ket{d_{e,f}})$ exist. 
If for every $\ket{d_{e,f}}$ and  every $\ket{c_{p,s,t}}$, one of the following conditions hlods:
 \begin{enumerate}
         \item[$(1)$] $u>0$ and $\ket{c_{p,s,t,[m,m+m_p-1]}}\neq \mathbf{0}_{m}$;
  \item[$(2)$] $u>1$ and the first $m_1+m_2$ components of $\ket{d_{e,f}}$ are all nonzero;
  \item[$(3)$] $m_1=\dots=m_{u_1}=1$, $m_{u_1+1}=\dots=m_{u}=m$, $1\leq u_1<m$, $m\mid w$ and $\ket{d_{e,f}}$ has $m$ consecutive nonzereo components starting at a position whose index is congruent to $0$ modulo $m$,
\end{enumerate}           
then for each $1 \leq p \leq u$, $D$ and $C_p$ are $(m,p;m_1,\dots,m_p,\dots,m_u)$-type disjoint.
\end{lemma}

\subsection{ Preliminaries}

In this subsection, we  attempt to obtain some results of quantum Latin squares, incomplete quantum Latin squares, and classical Latin squares needed in Theorems~\ref{MCDQLS-SDPC} and~\ref{MCIDQLS-SDPC} for later use.

It was pointed out in Ref.~\cite{cite18} that for any positive integers \( m \geq 2 \) and \( n \geq 2 \), there exist infinitely many pairwise disjoint MCQLS(\( mn \))s  or infinitely many pairwise disjoint  MCIQLS(\( mn,1 \))s. Note that the orders of the infinitely many disjoint MCQLSs here are composite numbers. Below, we first generalize this result to arbitrary orders.

\begin{lemma} \label{many MCIDQLS$(m+h;h)$} Let $m$ and $h$ be nonnegative integers. If there exists an MCIQLS$(m+h,h)$ such that for every nonempty entry, the sum of its first $m$ components is nonzero, then there exist infinitely many pairwise disjoint MCIQLS$(m+h,h)$s, each of which also has the property that the sum of the first $m$ components of every nonempty element is nonzero. 
Furthermore, the infinitely many pairwise disjoint MCIQLS$(m+h,h)$s have the same transversal properties as the initial MCIQLS$(m+h,h)$. 
\end{lemma}

\begin{proof}
Let $A_0$ be an  MCIQLS$(m+h,h)$ such that for every nonempty entry, the sum of its first $m$ components is nonzero.
If in $A_0$ the first $m$ components of some nonempty entry are all equal and nonzero,
then we can multiply the second component of each entry in $A_0$ by a suitable complex number $e^{i\lambda_0}$ with $\lambda_0 \in (0, 2\pi)$  to  ensure that the sum of the first $m$ components of every new entry is nonzero.  That is, we apply to $A_0$ the unitary transformation $\mathcal{U}$ corresponding to the unitary matrix $U = \operatorname{diag}(1, e^{i\lambda_0}, 1, \dots, 1)$ with respect to the computational basis. 
Since a unitary transformation is invertible and preserves inner products, and since both $\operatorname{span}\{\ket{0}, \ket{1}, \dots, \ket{m-1}\}$ and $\operatorname{span}\{\ket{m}, \ket{m+1}, \dots, \ket{m+h-1}\}$ are 
$\mathcal{U}$-subspaces, 
it follows that the resulting array $A=(\ket{a_{i,j}})$ remains an MCIQLS($m+h,h$) with the same transversal properties as before.

       Construct  new unitary matrices of the form
      \begin{align*}
            F_{m,h}^{\theta} = \begin{pmatrix}
                F_m^{\theta} & 0\\
                0 & I_{h}\\
            \end{pmatrix}, 
        \end{align*}
where $I_{h}$ is the $h \times h$ identity matrix, $F_m^{\theta}= \operatorname{diag}(e^{i\theta}, 1,1, \dots, 1)F_m$ with $F_m$ being the Fourier matrix of order $m$,
and $\theta \in [0, 2\pi)$. 
Then apply to $A$ the unitary transformations $\mathcal{F}_{v,h}^{\theta}$  corresponding to the unitary matrix $F_{v,h}^{\theta}$ with respect to the computational basis, and  denote the resulting MCIQLS$(m+h,h)$ by $A_{\theta}$. Similarly, $A_{\theta}$ is also an MCIQLS$(m+h,h)$ with the same transversal properties as before.

Consider the sum of the first $m$ components of any entry $F_{m,h}^{\theta}\ket{a_{i,j}}$ in $A_{\theta}$. Clearly, its $i$-th component is the product of the $i$-th row vector of $F_{m,h}^{\theta}$ with $\ket{a_{i,j}}$; therefore, the sum of the first $m$ components equals the product of the sum vector of the first $m$ row vectors of $F_{m,h}^{\theta}$ with $\ket{a_{i,j}}$. Note that when $s \in \{1,2,\dots, m-1\}$, we have $\sum\limits_{t=0}^{m-1} \omega^{st} = 0$. Consequently, the sum vector of the first $m$ row vectors of $F_{m,h}^{\theta}$ is
\[
\frac{1}{\sqrt{m}} \bigl( e^{i\theta} - 1 + m,\; e^{i\theta} - 1,\; \dots,\; e^{i\theta} - 1,\; 0,\; \dots,\; 0 \bigr),
\]
and its product with $\ket{a_{i,j}}$ equals the product of the vector
\(
\frac{1}{\sqrt{m}} \bigl( e^{i\theta} - 1 + m,\; e^{i\theta} - 1,\; \dots,\; e^{i\theta} - 1 \bigr)
\)
with $(a_0, a_1, \dots, a_{m-1})$, i.e.,
\[
 \tfrac{1}{\sqrt{m}}(e^{i\theta}-1) \sum_{i=0}^{m-1} a_i + \sqrt{m}\, a_0.
\]
Since $\sum\limits_{i=0}^{m-1} a_i$ is nonzero, one can always choose the value of $\theta$ 
(and there are infinitely many such choices) 
such that the sum of the first $m$ components of every nonempty entry of $A_{\theta}$ is nonzero.

We now show that one can always choose distinct $\theta_1, \theta_2 \in [0, 2\pi)$  such that the new MCQLS$(m+h,h)$ $A_{\theta_1}$ and $A_{\theta_2}$, obtained by acting on $A$ with the unitary transformations $\mathcal{F}_{v,h}^{\theta _1}$  and $\mathcal{F}_{v,h}^{\theta_2}$  respectively, are disjoint; and moreover, the sum of the first $m$ components of every nonempty entry in  $A_{\theta_1}$ and $A_{\theta_2}$ is nonzero. From the above analysis, it suffices to explain how the two arrays can be made disjoint.

Consider
\begin{align} \label{eq}
F_{m,h}^{\theta _1}\alpha=F_{m,h}^{ \theta _2 }\beta e^{i\lambda},
\end{align}
where $\theta_1, \theta_2, \lambda \in [0,2\pi)$ with $\theta_1 \neq \theta_2$ are to be determined, and $\alpha = (a_0, a_1, \dots, a_{m+h-1})^{T}$, $\beta = (b_0, b_1, \dots, b_{m+h-1})^{T}$ are two entries in $A$.

 When $\alpha = \beta$, we have $$F_{m,h}^{\theta _i}\alpha=\begin{pmatrix}F_{m}^{\theta _i}\alpha _{[0,m-1]}\\
 \alpha _{[m,m+h-1]}\end{pmatrix}$$ for $i = 1,2$. If $\alpha_{[m,m+h-1]} \neq \mathbf{0}_{h}$, then the equality (\ref{eq})
 implies $\lambda = 0$ and $F_{m}^{\theta_1}\alpha_{[0,m-1]} = F_{m}^{\theta_2}\alpha_{[0,m-1]} $.
 However, in the latter equality,
 the first components of the vectors on both sides are unequal when $\sum\limits_{i=0}^{m-1} a_i \neq 0$ and $\theta_1 \neq \theta_2$. 
If $\alpha_{[m,m+h-1]} = \mathbf{0}_{h}$ or $h=0$, then $\alpha_{[0,m-1]}\neq  \mathbf{0}_{m}$ and the equality $F_{m,h}^{\theta_1} \alpha = F_{m,h}^{\theta_2}\alpha e^{i\lambda}$ implies $\lambda = \theta_1 - \theta_2$ or $\lambda = 2\pi + \theta_1 - \theta_2$, together with $F_{m}^{\theta_1} \alpha_{[0,m-1]} = F_{m}^{\theta_2} \alpha_{[0,m-1]} e^{i\lambda}$. Note that in this case $\alpha_{[0,m-1]} \notin \operatorname{span}\{(1,1,\dots,1)^{T}\}$, while the $m$ row vectors of $F_{m}^0 = F_v$ form an orthonormal basis of $\mathcal{H}_{m}$. Consequently, the last $m-1$ components of $F_{m}^{\theta _1}\alpha _{[0,m-1]}$ (equivalently, of $F_{m}^{\theta _2}\alpha _{[0,m-1]}$ ) cannot all be zero. Equating these two vectors again yields $\lambda = 0$, a contradiction.

When $\alpha\neq \beta$, 
 multiplying both sides of the equation (\ref{eq}) on the left by the conjugate transpose ${F_{m,h}^{\theta _2}}^{\dagger}$ of ${F_{m,h}^{\theta _2}}$, we obviously obtain $${F_{m,h}^{\theta _2}}^{\dagger}F_{m,h}^{\theta _1}\alpha=\beta e^{i\lambda}.$$
       Write $F_{m,h}^{\theta_1,\theta_2} = (F_{m,h}^{\theta_2})^\dagger F_{m,h}^{\theta_1}$, $F_{m}^{\theta_1,\theta_2} = (F_{m}^{\theta_2})^\dagger F_{m}^{\theta_1}$.
 Note that
\begin{align*} & F_m^{(\theta _1,\theta _2) } = {F_{m}^{\theta _2}}^{\dagger}F_{m}^{\theta _1}\\
        =&\frac{1}{m}\begin{pmatrix}
            e^{i(\theta_1-\theta_2)} + m - 1 & e^{i(\theta_1-\theta_2)} - 1 & \dots & e^{i(\theta_1-\theta_2)} - 1 & \dots & e^{i(\theta_1-\theta_2)} - 1\\
            e^{i(\theta_1-\theta_2)} - 1 & e^{i(\theta_1-\theta_2)} + m - 1 & \dots & e^{i(\theta_1-\theta_2)} - 1 & \dots & e^{i(\theta_1-\theta_2)} - 1\\
            \vdots & \vdots &  & \vdots &  & \vdots\\
            e^{i(\theta_1-\theta_2)} - 1 & e^{i(\theta_1-\theta_2)} - 1 & \dots & e^{i(\theta_1-\theta_2)} + m - 1 & \dots & e^{i(\theta_1-\theta_2)} - 1\\
            \vdots & \vdots &  & \vdots &  & \vdots\\
            e^{i(\theta_1-\theta_2)} - 1 & e^{i(\theta_1-\theta_2)} - 1 & \dots & e^{i(\theta_1-\theta_2)} - 1 & \dots & e^{i(\theta_1-\theta_2)} + m - 1
        \end{pmatrix}\\
        =&\frac{1}{m} (e^{i(\theta_1-\theta_2)} - 1) J_m + I_m.
        \end{align*}where $J_m$ is the $m \times m$ all-ones matrix.
   So we have that
    \begin{align*}
    F_{m,h}^{(\theta_1, \theta_2)} &= \begin{pmatrix}
        {F_m^{\theta_2}}^{\dagger} & 0\\
        0 & I_h\\
    \end{pmatrix}
    \begin{pmatrix}
        F_m^{\theta_1}  & 0\\
        0 & I_h\\
    \end{pmatrix}\\
    &= \begin{pmatrix}
        {F_m^{\theta_2}}^{\dagger} F_m^{\theta_1} & 0\\
        0 & I_h\\
    \end{pmatrix}\\
    &= \begin{pmatrix}
        \frac{1}{m}(e^{i(\theta_1-\theta_2)} - 1) J_m + I_m & 0\\
        0 & I_h\\
    \end{pmatrix}.
    \end{align*}
    Hence $\begin{pmatrix}
        \frac{1}{m}(e^{i(\theta_1-\theta_2)} - 1) J_m + I_m & 0\\
        0 & I_h\\
    \end{pmatrix}
    \begin{pmatrix}
        \alpha_{[0,m-1]}\\ \alpha_{[m,m+h-1]} 
    \end{pmatrix} = e^{i \lambda} \begin{pmatrix}
        \beta_{[0,m-1]}\\ \beta_{[m,m+h-1]}
    \end{pmatrix}$, i.e.,
    \begin{align*}
    \begin{cases}
        (\frac{1}{m}(e^{i(\theta_1-\theta_2)} - 1) J_m + I_m) \alpha_{[0,m-1]} = e^{i \lambda}  \beta_{[0,m-1]},\\
        \alpha_{[m,m+h-1]} = e^{i \lambda} \beta_{[m,m+h-1]}.
    \end{cases}
    \end{align*}
    or
    \begin{align}
    \begin{cases} \label{(1a)}
        \frac{1}{m}(e^{i(\theta_1-\theta_2)} - 1) \displaystyle\sum_{i=0}^{m-1}a_i\begin{pmatrix}
        1\\
        \vdots \\ 1
    \end{pmatrix} 
    = e^{i \lambda} \beta_{[0,m-1]} - \alpha_{[0,m-1]},\\
         \alpha_{[m,m+h-1]} = e^{i \lambda} \beta_{[m,m+h-1]}.
    \end{cases}
    \end{align}
  Since $\sum\limits_{i=0}^{m-1} a_i \neq 0$, Equation (\ref{(1a)}) implies that all components of the vector  $e^{i \lambda} \beta_{[0,m-1]} - \alpha_{[0,m-1]}$ are equal. 
In this case, if $\alpha_{[m,m+h-1]} = \beta_{[m,m+h-1]} = \mathbf{0}_{h}$ or $h = 0$, and the components of $\beta_{[0,m-1]}\notin \operatorname{span}\{(1,1,\dots,1)^{T}\}$, we may assume without loss of generality that $b_0 \neq b_1$. Then the equality \(e^{i\lambda}b_0 - a_0 = e^{i\lambda}b_1 - a_1\) determines at most one $\lambda$ that can possibly make the first equality in Equation \ref{(1a)}  hold. However, even if such a $\lambda$ exists, once $e^{i\lambda}\beta_{[0,m-1]} - \alpha_{[0,m-1]}$ is fixed, one can always choose $\theta_1, \theta_2$ appropriately so that the first equality in  Equation \ref{(1a)}   does not hold. 
Finally, if both $\alpha_{[0,m-1]}$ and $\beta_{[0,m-1]}$ are nonzero vectors, then there is at most one $\lambda$ such that  the second equality in  Equation \ref{(1a)}   holds. After such a $\lambda$ is determined, one can again adjust the values of $\theta_1, \theta_2$ appropriately so that the first equality in  \ref{(1a)} fails. 

Consequently, by choosing $\theta_1$ and $\theta_2$ appropriately,  $F_{m,h}^{\theta_1}\alpha$ and $F_{m,h}^{\theta_2}\beta$ are distinct for any two entries $\alpha$ and $\beta $ in  $A$. 
Then, we try to obtain more pairwise disjoint MCIQLS$(m)$ with the sum of the first $m$ components of every nonempty entry in each of these arrays is nonzero. Similar to the previous analysis, starting from $A$, we can find a suitable $\theta_3 \in [0, 2\pi) \setminus \{\theta_1, \theta_2\}$ such that $A_{\theta_3}$ is disjoint from both $A_{\theta_1}$ and $A_{\theta_2}$, and the sum of the first $m$ components of every entry in $A_{\theta_3}$ is nonzero.
Continuing this process, we can select infinitely many distinct parameters $\theta_1, \theta_2, \dots, \theta_i, \dots$ such that the corresponding $A_{\theta_1}, A_{\theta_2}, \dots, A_{\theta_i}, \dots$ are pairwise disjoint MCIQLS$(m)$s.
Clearly,  $A_{\theta_1}, A_{\theta_2}, \dots, A_{\theta_i}, \dots$ have the same transversal  properties as $A$ and $A_0$.
\qed
\end{proof}

\begin{remark}
When $h=1$, suppose we have obtained finitely many pairwise disjoint MCIQLS$(m+1,1)$s such that for each of them, the sum of the first $m$ components is nonzero. Then one can always find a complex number $y$ with $|y|=1$ such that multiplying the last component of every nonempty entry of these MCIQLS$(m+1,1)$s by $y$ yields MCIQLS$(m+1,1)$s that are still pairwise disjoint. Moreover, after this multiplication, for each entry, both the sum of its first $m$ components and the sum of all its components are nonzero.-

\end{remark}

\begin{lemma}\label{disjoint MCQLS$(v)$}
For any integer $m\geq 4$, there exist infinitely many pairwise disjoint MCQLS$(m)$s. 
\end{lemma}
\begin{proof}
Observe the  MCQLS$(m)$ $A=(\ket{a_{i,j}})$ constructed in \cite{cite19}, where $$\ket{a_{i,j}}=\frac{1}{\sqrt{m}}(1,\omega^{2i+j}, \omega^{i+2j}, \omega^{3(i+j)}, \dots, \omega^{(m-1)(i+j)})^T,$$  $\omega$ is a primitive $m$-th root of unity.
Note that all components of every $\ket{a_{i,j}}$ are nonzero. If some entry in $A$ has a zero sum of components, we may multiply the first component of every entry of $A$ by a suitable complex number $x$ with $|x|=1$ and $x \neq 1$. That is, applying the unitary transformation corresponding to the unitary matrix $\operatorname{diag}(x, 1, \dots, 1)$ to $A$ yields a new MCQLS$(m)$  in which the sum of components of each entry is nonzero. Then applying Lemma \ref{many MCIDQLS$(m+h;h)$}, we can obtain infinitely many pairwise disjoint MCQLS$(m)$. \qed 
\end{proof}

\begin{lemma}\label{disjoint MCIDQLS$(v)$}
For any odd integer $m \ge 7$, or for any even square $m \ge 16$, there exist infinitely many pairwise disjoint idempotent MCQLS$(m)$s.
Furthermore, for any odd $m$ with $m \geq 13$ and $3 \nmid m$, or for any even square $m \geq 16$, there exist infinitely many pairwise disjoint MCDQLS$(m)$s.
\end{lemma}
\begin{proof}
 When $m$ is odd and $m \ge 7$, observe the  idempotent  MCQLS$(m)$  constructed in Lemma \ref{MCIDQLS-odd}. Clearly  the first component of every entry  is $1$. Similar to the discussion in the proof of  Lemma \ref{disjoint MCQLS$(v)$}, we can obtain infinitely many pairwise disjoint idempotent  MCQLS$(m)$s by applying Lemma \ref{many MCIDQLS$(m+h;h)$}.
 
 When $m$ is an even squareand $m \ge 16$, Observe the idempotent MCQLS$(m)$ constructed in Lemma~\ref{MCIDQLS(v^2)}, whose entries are all tensor products of the entries of the MCQLS$(m)$ constructed in \cite{cite19}. Clearly  the first component of every entry  is $1$. we can also obtain infinitely many pairwise disjoint idempotent  MCQLS$(m)$s by applying Lemma \ref{many MCIDQLS$(m+h;h)$}.\qed 
 
 Similarly,  the conlusion  also holds for MCDQLS$(m)$ by  Lemma~\ref{MCIDQLS(v^2)}.
lemma \ref{}, we can obtain infinitely many pairwise 
\end{proof}

\begin{lemma}\label{disjoint MCIDQLS$(v+1;1)$}
For any integer $m\geq 3$, there exist infinitely many pairwise disjoint MCIQLS$(m+1,1)$s.
\end{lemma}
\begin{proof}
Observe the  MCQLS$(m+1)$ $A=(\ket{a_{i,j}})$ constructed in \cite{cite19}, where $$\ket{a_{i,j}}=\frac{1}{\sqrt{m+1}}(1,\omega^{2i+j}, \omega^{i+2j}, \omega^{3(i+j)}, \dots, \omega^{m(i+j)})^T,$$ $\omega$ is a primitive $(m+1)$-th root of unity. The entries in its first row are exactly the $m+1$ column vectors of the Fourier matrix $F_{m+1}$:
$$f_j=\ket{a_{i,j}}=\frac{1}{\sqrt{m+1}}(1,\omega^{j}, \omega^{2j}, \omega^{3j}, \dots, \omega^{mj})^T, j=0,1,\dots,m.$$
Applying the unitary transformation corresponding to $F_{m+1}^{-1} = F_{m+1}^{\dagger}$ to $A$ yields MCQLS$(m+1)$ $A_1 = (\ket{a_{i,j}'}) = (F_{m+1}^{-1}\ket{a_{i,j}})$. Clearly, $\ket{a_{0,m}'} = F_{m+1}^{-1} f_{m} = \ket{m}$, and
    \begin{align*}
        \ket{a_{i,j}'} = \frac{1}{m+1}
        \begin{pmatrix}
            1 & 1 & \dots & 1 & \dots & 1\\
            1 & \omega^{-1} & \dots & \omega^{-i} & \dots & \omega^{-m}\\
            \vdots & \vdots &  & \vdots &  & \vdots\\
            1 & \omega^{-j} & \dots & \omega^{-ij} & \dots & \omega^{-mj}\\
            \vdots & \vdots &  & \vdots &  & \vdots\\
            1 & \omega^{-m} & \dots & \omega^{-im} & \dots & \omega^{-m^2}\\
        \end{pmatrix} 
        \begin{pmatrix}1\\ \omega^{2i+j}\\ \omega^{i+2j}\\\omega^{3(i+j)} \\ \vdots \\ \omega^{m(i+j)}\end{pmatrix}.
    \end{align*}
    Note that when $i+j \neq 0$, $\sum\limits_{t=0}^{m} \omega^{t(i+j)} = 0$. Then for any $(i,j) \neq (0,0)$, the first component of $\ket{a_{i,j}'}$ is
    \[
  \frac{1}{m+1}(  1 + \omega^{2i+j} + \omega^{i+2j}+ \omega^{3(i+j)} + \dots + \omega^{m(i+j)}) = \frac{1}{m+1}\omega^{i+j}(\omega^i-1)(1-\omega^{j}).
    \]
Thus, only when $i = 0$, or $j = 0$ and $(i,j) \neq (0,0)$, the first component of $\ket{a_{i,j}'}$ is zero. 
And when $i \neq 0$, we have
    \begin{align*}
    \ket{a_{i,0}'} &= \frac{1}{m+1}\begin{pmatrix}
        0\\ \sum_{t=3}^{m} \omega^{t(i-1)}+1+\omega^{2i-1}+\omega^{i-2}\\ \vdots\\ \sum_{t=3}^{m} \omega^{t(i-k)}+1+\omega^{2i-k}+\omega^{i-2k}\\ \vdots\\ \sum_{t=3}^{m} \omega^{t(i-m)}+1+\omega^{2i-m}+\omega^{i-2m}
    \end{pmatrix}\\
   & = \frac{1}{m+1}\begin{pmatrix}
        0\\ 1+\omega^{2i-1}+\omega^{i-2}-1-\omega^{i-1}-\omega^{2(i-1)}\\ \vdots\\ 1+\omega^{2i-k}+\omega^{i-2k}-1-\omega^{i-k}-\omega^{2(i-k)}\\ \vdots\\ 1+\omega^{2i-m}+\omega^{i-2m}-1-\omega^{i-m}-\omega^{2(i-m)}
    \end{pmatrix}\\
    &= \frac{1}{m+1}\begin{pmatrix}
        0\\ \omega^{i-2}(\omega^i-1)(\omega-1)\\ \vdots\\ \omega^{i-2k}(\omega^i-1)(\omega^{k}-1)\\ \vdots\\ \omega^{i-2m}(\omega^i-1)(\omega^{m}-1)
    \end{pmatrix}. 
    \end{align*}
The sum of the first $m$ components of $\ket{a_{i,0}'}$ is given by
\begin{align*}
&\frac{1}{m+1}(\omega^i-1)\sum_{k=1}^{m-1} \omega^{i-2k}(\omega^{k}-1)\\
=& \frac{1}{m+1} \omega^{i}(\omega^i-1)\sum_{k=1}^{m-1} (\omega^{-k}-\omega^{-2k})\\
=& \frac{1}{m+1}\omega^{i}(\omega^i-1) \bigl[-1-\omega^{-m}-(-1-\omega^{-2m})\bigr]\\
= &\omega^{i-2m}(\omega^i-1)(1-\omega^{m}) \neq 0.
\end{align*}

   Applying the unitary transformation corresponding to the matrix 
\(U = \begin{pmatrix} e^{ix} & \\ & I_m \end{pmatrix}\) 
(with \(x \in [0, 2\pi)\) to be determined) to \(A_1\) yields an 
\((m+1)\)-dimensional MCQLS matrix \(A_2 = \bigl(\ket{a_{i,j}''}\bigr)\), 
where \(\ket{a_{i,j}''} = U\ket{a_{i,j}'}\). In this new MCQLS$(m+1)$,  $\ket{a_{0,m}''} = \ket{m}$, and when  $i = 0$, or $j = 0$ and $(i,j) \neq (0,0)$, the first component of $\ket{a_{i,j}''}$ is zero while the sum of the first $m$ components of $\ket{a_{i,j}''}$ is nonzero. when  $i,j \neq 0$, the first component of  $\ket{a_{i,j}''}$ is $e^{ix}$.  One can always choose $x$ appropriately so that the sum of the first $m$ components of every element in $A_2$ is nonzero.
By adjusting the rows and columns of $A_2$, we obtain an  MCQLS$(m+1)$ whose lower-right element is $\ket{m}$, and for the remaining entries, the sum of the first $m$ components is nonzero. After $\ket{m}$ becomes an empty entry, we obtain an MCIQLS$(m+1,1)$,
in which the sum of the first $m$ components of every non-empty entry is nonzero.  Finally,   by applying Lemma \ref{many MCIDQLS$(m+h;h)$}, we can obtain infinitely many pairwise disjoint MCIQLS$(m+1,1)$.\qed
\end{proof}

\begin{lemma}\label{disjoint MCQLS$(m)$+ MCIDQLS$(m+1;1)$}
Let $s$ and $t$ be positive integers, and let $m$ be an integer with $m \ge 4$ (or $m \ge 7$ odd, or $m\ge 16$ an even square). There exist $s$ pairwise disjoint MCIQLS$(m+1; 1)$s,  $t$ pairwise disjoint  MCQLS$(m)$s (or idempotent MCQLS$(m)$s), and moreover these  MCQLS$(m)$s (or idempotent MCQLS$(m)$s) are also $m$-disjoint from the $s$  MCIQLS$(m+1, 1)$s.
\end{lemma}
\begin{proof}
By Lemma \ref{disjoint MCIDQLS$(v+1;1)$}, there exist $s$ pairwise disjoint  MCIQLS$(m+1, 1)$s. 
Let $S$ denote the set of projection vectors obtained by projecting all entries contained in the $s$ MCIQLS$(m+1, 1)$s onto the first $m$ coordinates. Then $|S| \leq s \cdot [(m+1)^2 - 1]$. 
By Lemma \ref{disjoint MCQLS$(v)$} (or Lemma \ref{disjoint MCIDQLS$(v)$}),  there are infinitely many pairwise disjoint  MCQLS$(m)$s (or idempotent MCQLS$(m)$s). Select $t + s[(m+1)^2 - 1]$  of them. 
By the pigeonhole principle, at most $s[(m+1)^2 - 1]$ pairwise disjoint MCQLS$(m)$s (or idempotent  MCQLS$(m)$s) contain some vector belonging to $S$.
Therefore, the remaining $t$ disjoint pairwise disjoint  MCQLS$(m)$s (or idempotent MCQLS$(m)$s) are pairwise $m$-disjoint from the $s$ MCIQLS$(m+1, 1)$s.\qed 
\end{proof}

\begin{lemma}\label{disjoint MCIDQLS$(2m;m)$}
    For any integer $n \geq 2$ and $m \geq 4$, there exist infinitely many pairwise disjoint MCIQLS$(nm, m)$. 
\end{lemma}
\begin{proof}
For $n \geq 2$, there exists an ILS$(n,1)$ whose last row or last column is a permutation of $\mathbb{Z}_n\setminus \{n-1\}$. Replacing each entry $i$ in this ILS$(n, 1)$ by $\ket{i}$ ($i \in \mathbb{Z}_n$), we obtain an IQLS$(n, 1)$, denoted by $A = (\ket{a_{i,j}})$. By lemma \ref{disjoint MCQLS$(v)$}, there exist infinitely many  pairwise disjoint MCIQLS$(m)$
Thus we can take $n$ pairwise disjoint MCQLS$(m)$ 
    $ B_t= (\ket{b_{t,k,l}})$, $t\in\mathbb{Z}_n$, and replace each nonempty entry $\ket{a_{i,j}}$ in $A$ by $\ket{a_{i,j}} \otimes B_i$, $i,j\in \mathbb{Z}_n$, and replace each empty entry by an empty $m \times m$ array. Similar to the proof of Theorem~\ref{MCDQLS-SDPC}, we obtain an MCIQLS$(nm, m)$, denoted by $N$.
    
Continue to select $n$ pairwise disjoint MCQLS$(m)$s $B'_t = (\ket{b'_{t,i,j}})$, $t \in \mathbb{Z}_n$, each of which is disjoint from $B_0, B_1, \dots, B_{n-1}$. Repeating the previous construction, replace each nonempty entry $\ket{a_{i,j}}$ in $A$ by  $\ket{a_{i,j}} \otimes B'_i$ for $i, j \in \mathbb{Z}_n$ to obtain an MCIQLS$(nm, m)$, denoted by $N'$. Since for any $i,j,i',j' \in \mathbb{Z}_n$, $k, l, k', l' \in \mathbb{Z}_m$, whenever $(i,j,k,l) \neq (i',j',k',l')$ we have $\ket{b_{i,k,l}}\neq  \ket{b'_{i',k',l'}}$, it follows that \[
    \ket{a_{i,j}} \otimes \ket{b_{i,k,l}}\neq \ket{a_{i',j'}} \otimes \ket{b'_{i',k',l'}},
    \] by Lemma \ref{state vector},  and thus $N$ and $N'$ are disjoint.
    
Continuing the above process, we can eventually obtain infinitely many desired pairwise disjoint\ MCIQLS$(nm, m)$s. \qed
\end{proof}
\begin{remark}
In the proof of Lemma~\ref{disjoint MCIDQLS$(2m;m)$}, it is clear that if each entry in $B_0, B_1, \dots, B_{n-1}; B'_0, B'_1, \dots, $ $B'_{n-1}; \dots$ has a nonzero sum of its components, then each entry of the resulting infinitely many disjoint MCIQLS$(nm, m)$s also has a nonzero sum of its components. Take $h=0$ in Lemma~\ref{many MCIDQLS$(m+h;h)$}, then from the proof of Lemma~\ref{disjoint MCQLS$(v)$}, we know that there exist infinitely many disjoint MCQLS$(m)$s with $m\geq 4$ and and  each entry has a nonzero sum of its components. Thus, there exist infinitely many pairwise disjoint MCIQLS$(nm, m)$s and  each entry has a nonzero sum of its components.

 \end{remark}
\begin{lemma}\label{disjoint MCQLS$(m)$+ MCIDQLS$(m+1;1)$+MCIDQLS$(2m;m)$}
Let $k$, $s$, $t$ be arbitrary integers, and and let $m$ be an integer with $m \ge 4$ (or $m \ge 7$ odd, or $m\ge 16$ an even square). There exist $k$ pairwise disjoint MCIQLS$(2m, m)$s, $s$ pairwise disjoint MCIQLS$(m+1, 1)$s,  $t$ pairwise disjoint MCQLS$(m)$s (or idempotent MCQLS$(m)$s), and moreover, every  MCQLS$(m)$s (or  idempotent MCQLS$(m)$s), MCIQLS$(m+1, 1)$s, and  MCIQLS$(2m, m)$s  are pairwise  $m$-disjoint.
\end{lemma}
\begin{proof} Since an ILS$(2,1)$ exists, by Lemma~\ref{disjoint MCQLS$(v)$}, there exist $k$ pairwise disjoint MCIQLS$(2m,m)$s.

Let $S_1$ denote the set of projection vectors obtained by projecting all elements contained in the $k$ MCIQLS$(2m, m)$s onto the first $m$ coordinates. Then $\lvert S_1 \rvert \leq 3km^2$.
By Lemma \ref{disjoint MCIDQLS$(v+1;1)$},  there are infinitely many pairwise disjoint  MCIQLS$(m+1,1)$s. Select $s +  3km^2[(m+1)^2-1]$  of them. 
By the pigeonhole principle, There are at most $3km^2[(m+1)^2-1]$ pairwise disjoint MCIQLS$(m+1,1)$s having a projection vector onto the first $m$ coordinates equal to some vector of $S_1$.
Therefore, the remaining $s$ disjoint pairwise disjoint  MCIQLS$(m+1,1)$s  are pairwise disjoint from the $k$ MCIQLS$(2m, m)$s.

Let $S_2$ denote the set of projection vectors obtained by projecting all elements contained in the $k$ MCIQLS$(2m, m)$ and the  $s$   MCIQLS$(m+1,1)$s onto the first $m$ coordinates. Then $\lvert S_2 \rvert \leq 3km^2+s[(m+1)^2-1]$. By Lemma \ref{disjoint MCQLS$(v)$} (or Lemma \ref{disjoint MCIDQLS$(v)$}),  there are infinitely many pairwise disjoint  MCQLS$(m)$s (or idempotent MCQLS$(m)$s).
 Select $t +  3km^2+s[(m+1)^2-1]$   MCQLS$(m)$s (or idempotent  MCQLS$(m)$s ).
By the pigeonhole principle, There are $t$ disjoint pairwise disjoint   disjoint pairwise disjoint  MCQLS$(m)$s (or idempotent MCQLS$(m)$s) $m$-disjoint from the previous $s$ MCIQLS$(m+h, h)$s and the previous $k$ MCIQLS$(2m, m)$s. \qed
\end{proof}

\begin{remark} \label{sum}
Similar to the previous analysis, we can guarantee  that each entry of the the $k$  MCIQLS$(2m, m)$s, $s$  MCIQLS$(m+1, 1)$s and  $t$ MCQLS$(m)$s (or idempotent MCQLS$(m)$s) in Lemma \ref{disjoint MCQLS$(m)$+ MCIDQLS$(m+1;1)$+MCIDQLS$(2m;m)$} has a nonzero sum of its components.
\end{remark}

We introduce the definition of difference matrices and use them to construct some more Latin squares with specific transversal conditions.

Let $(G, \odot)$ be a group of order $v$. A $(v, k; \lambda)$-difference matrix (or $(v, k; \lambda)$-DM) is a $k \times v\lambda$ matrix $D = (d_{i,j})$ with entries from $G$, such that for each $1 \leq t < h \leq k$, the multiset
\[
\{d_{t, \ell} \odot d_{h, \ell}^{-1} : 1 \leq \ell \leq v\lambda\}
\]
contains every element of $G$ exactly $\lambda$ times.

\begin{lemma}\cite{cite21} \label{lem3.2.3}
   There exists a $(v, 4; 1)$-$\mathrm{DM}$ if and only if $v \geq 4$ and $v \not\equiv 2 \pmod{4}$. 
\end{lemma}

\begin{lemma}\label{DLS}
For any positive integer  $n$ with $4\mid n$, there exists a $\mathrm{DLS}(n)$  that has $n$ pairwise disjoint transversals, including the main diagonal and the anti-diagonal.
\end{lemma}
\begin{proof}
By Lemma \ref{lem3.2.3}, there exists a $(n, 3; 1)$-DM if $n \equiv 0 \pmod{4}$.  Without loss of generality, assume that $D=(d_{ij})$ is an $(n,3,1)$-DM over an Abelian group $(G,\odot)$, and every entry of its first row is the identity element. Further assume that the entries of its second row are $g_1=\theta, g_2, \dots, g_{\frac{n}{2}}, g_{\frac{n}{2}+1}, \dots, g_{n-1}, g_n$, with $g_{n+1-i}=g_i\odot \alpha$, where $\alpha$ is an element of order $2$ in $G$. By the proof of Theorem 2.1 of Ref.~\cite{cite}, we can construct two orthogonal Latin squares $L^1$  and $L^2$, where $L^{l}(i,j)=g_i^{-1}\odot d_{l+1,j}$, $l=1,2$. Note that every entry  on the main diagonal of $L^{1}$ is $\theta$, and every entry  on the anti-diagonal of $L^{1}$ is $\alpha$, so $L^{2}$ is a DLS$(n)$. 
Observe the positions in $L^{2}$ whose entries are all $g_i^{-1}\odot d_{2j}=g_i^{-1}\odot g_{j}=L^{1}(i,j)$.
If $L^{1}(i',j') = g_{i'}^{-1} \odot g_{j'} = g_i^{-1} \odot g_j$ with $(i,j) \neq (i',j')$, then we must have $i \neq i'$ and $j \neq j'$. Otherwise, assume without loss of generality that $i = i'$;
then $g_i^{-1} = g_{i'}^{-1}$, and hence $g_j = g_{j'}$, 
which implies $j = j'$.
Thus, the element $g_i^{-1} \odot g_j$ in $L^{1}$ appears in different rows and columns of $L^{1}$, which corresponds to the fact that $L^{2}$ produces a transversal in those corresponding positions. Finally, $L^{2}$ is a $\mathrm{DLS}(n)$  that has $n$ pairwise disjoint transversals, including the main diagonal and the anti-diagonal.\qed

\end{proof}

\begin{lemma}  \label{DLs-6,10,14}
    If $n\in \{6,10,14\}$, there exists a $\mathrm{DLS}(n)$ that has $n-2$ pairwise disjoint transversals, including the main diagonal and the anti-diagonal.
\end{lemma}
\begin{proof}
The desired  $\mathrm{DLS}(n)$s are given in Appendix B. \qed

\end{proof}

  \section{Existence results}
  
  In this section, we aim to determine the existence of an \(\mathrm{MCDQLS}(v)\). To this end, we first establish the existence of an idempotent \(\mathrm{MCQLS}(v)\).

\begin{theorem} \label{idempotent MCQLS$(v)$ -result-odd}
For any   integer \( v \geq 6 \), there   exists an idempotent MCQLS$(v)$ except possibly when \( v \in  
\{6,8,10,12,14,18,20,24,26,30,32, 38,62\} \).
\end{theorem}

\begin{proof}
By Lemma~\ref{MCIDQLS-odd}, it suffices to consider the even orders greater than or equal to $6$.

For any even positive integer \( v \geq 154 \), take  $m =7$ and suppose that $n \geq 9$ be odd. By Lemma~\ref{lem2.2.5}, there exists an idempotent LS$(n)$ containing $n$ disjoint transversals, including the main diagonal. By  Lemma \ref{disjoint MCQLS$(m)$+ MCIDQLS$(m+1;1)$},  there exist $u$ ($7 \leq u \leq n-2$) pairwise disjoint MCIQLS$(m+1, 1)$s,  $(n+2)$ pairwise disjoint  idempotent MCQLS$(m)$s, and  these   MCQLS$(m)$s are also $m$-disjoint from the $s$  MCIQLS$(m+1, 1)$s. Then applying Theorem \ref{MCIDQLS-SDPC} and Lemma \ref{type disjoint}, an  idempotent MCQLS$(mn+u)$  exists.
Furthermore, it follows that for any odd integer $n \geq 9$ and any integer $v \in [7n+7, 7n+n-2]_e = [7n+7, 8n-2]_e$, an idempotent MCQLS$(v)$ exists. Since $n+2$ is also odd, for any odd integer $n \geq 9$ and any integer $v \in [7(n+2)+7, 8(n+2)-2]_e$, an idempotent MCQLS$(v)$ also exists. Computing

\[
7(n+2)+7-2 \leq 8n-2
\]
yields $n \geq 21$. Hence for $v \geq 7 \times 21 + 7 = 154$, an idempotent MCQLS$(v)$ exists for all even orders $v$.

 Similarly, for any odd numbers $m \geq 7$, $n \geq 3$, $u = 1$ or $m \geq 7$, $n \geq 9$, $7 \leq u \leq n-2$, there exists an idempotent MCQLS$(mn+u)$ of  even order  by Theorem \ref{MCIDQLS-SDPC},  Lemma \ref{type disjoint}, \ref{disjoint MCQLS$(m)$+ MCIDQLS$(m+1;1)$} and \ref{lem2.2.5}.  
The Table below shows the orders for which an idempotent MCQLS$(mn+u)$ exists under the above processing method.

$$
\begin{tabular}{ccc} \label{idempotent MCQLS-Tabel1}
 $v$ & $n \times m + u$ & $u$  \\
 \hline  
           $[140, 150]_e$ & $19 \times 7 + s$ & $s \in [7,17]_o$   \\
           $[126, 134]_e$ & $17 \times 7 + s$ & $s \in [7,15]_o $  \\
           $[112, 118]_e$ & $15 \times 7 + s$ & $s \in [7,13]_o$  \\
           $\{98, 100, 102\}$ & $13 \times 7 + s$ & $s = 7,9,11$  \\
            $106$ & $9 \times 11 + 7$ & $ 7$  \\
 $94$ & $3 \times 31 + 1$ & $1$  \\
           $86$ & $7 \times 11 + 9$ & $ 9$  \\
           $82$ & $9 \times 9 + 1$ & $1$  \\
 $78$ & $11 \times 7 + 1$ & $ 1$ \\
                     $58$ & $3 \times 19 + 1$ & $1$  \\
           $50$ & $7 \times 7 + 1$ & $1$  \\
           $46$ & $3 \times 15 + 1$ & $1$  \\
           $40$ & $3 \times 13 + 1$ & $1$  \\
           $34$ & $3 \times 11 + 1$ & $1$  \\
           $22$ & $3 \times 7 + 1$ & $1$ \\
           \hline
    \end{tabular}
$$

For any integer $n \geq 3$, and for $m$ such that either $m \geq 7$ is odd or $m \geq 16$ is an even perfect square, there exists an idempotent MCQLS$(mn)$  by Theorem \ref{MCIDQLS-SDPC},  Lemma \ref{disjoint MCIDQLS$(v)$} and Lemma \ref{lem2.2.5}. 
Below, we list the orders for which an  idempotent MCQLS$(mn)$ exists under the above processing method.
$$\begin{tabular}{cccc}
        $152 = 8 \times 19$ & $104 = 8 \times 13$ & $72 = 8 \times 9$ & $52 = 4 \times 13$ \\
        $138 = 6 \times 23$ & $96 = 6 \times 16$ & $70 = 10 \times 7$ & $48 = 3 \times 16$ \\
        $136 = 8 \times 17$ & $92 = 4 \times 23$ & $68 = 4 \times 17$ & $44 = 4 \times 11$ \\
        $124 = 4 \times 31$ & $90 = 10 \times 9$ & $66 = 6 \times 11$ & $42 = 6 \times 7$ \\
        $122 = 6 \times 21$ & $88 = 8 \times 11$ & $64 = 8^2$ & $36 = 6^2$ \\
        $120 = 8 \times 15$ & $84 = 4 \times 21$ & $60 = 4 \times 15$ & $28 = 4 \times 7$ \\
        $110 = 10 \times 11$ & $80 = 5 \times 16$ & $56 = 8 \times 7$ & $16 = 4^2$ \\
        $108 = 4 \times 27$ & $76 = 4 \times 19$ & $54 = 6 \times 9$ & \\
    \end{tabular}
$$
 Besides, for $74 = 7 \times 9 + (9 + 2)$, take $n = 7$, $m = 9$, $m_1 = 9$, $m_2 = m_3 = 1$. There exists an idempotent MCQLS$(74)$  by Theorem \ref{MCIDQLS-SDPC},  Lemma \ref{type disjoint}, \ref{disjoint MCQLS$(m)$+ MCIDQLS$(m+1;1)$+MCIDQLS$(2m;m)$} and \ref{lem2.2.5}. 
 
   The proof is then completed. \qed

\end{proof}

\begin{remark}\label{MCIDQLS-R1}
From Remark \ref{sum--SDPC} and Remark \ref{sum}, we know that the idempotent MCQLS obtained in the proof of Theorem \ref{idempotent MCQLS$(v)$ -result-odd} can also satisfy that each entry has a nonzero sum of its components. Thus, by Lemmas \ref{many MCIDQLS$(m+h;h)$} and \ref{disjoint MCIDQLS$(v)$}, for \( v \geq 6 \), we can also obtain infinitely many disjoint idempotent MCQLS\((v)\) except possibly when \( v \in \{6,8,10,12,14,18,20,24,26,30,32,38,62\}\).
\end{remark}

\begin{theorem}\label{MCDQLS-result-odd}
For any odd  integer \( v \geq 7 \), there   exists an MCDQLS$(v)$ except possibly when \( v \in \{7, 9, 11, 15, 21, 27, 33, 39, 51, 63, 75\} \).
\end{theorem}
\begin{proof}
By Lemma \ref{MCDQLS-odd}, it suffices to take $v \geq 7$ and $3 \mid v$.  Set $v = 3a$, where $a\geq 3$ is an odd integer.

Let $n = 4s$ where $s$ is a positive integer. By lemma \ref{DLS}, there exists a DLS$(n)$ containing $n$ disjoint transversals, including the main diagonal and the anti-diagonal. 
Take $m = 7$.  By Lemma \ref{disjoint MCQLS$(m)$+ MCIDQLS$(m+1;1)$+MCIDQLS$(2m;m)$}, there exist $n-u$ pairwise disjoint idempotent MCQLS$(m)$, $u_1$ pairwise disjoint MCIQLS$(m+1,1)$, and $u-u_1$ pairwise disjoint MCIQLS$(2m,m)$ with $0 \leq u \leq n-2$ and $0 \leq u_1 \leq u$, such that all these MCQLS$(m)$s, MCIQLS$(m+1,1)$s, and MCIQLS$(2m,m)$s are pairwise $m$-disjoint. Arrange the $u_1$ MCIQLS$(m+1,1)$ and the $u-u_1$ MCIQLS$(2m,m)$ in order, set $m_p = 1$ if $p \leq u_1$, and $m_p = m$ if $p > u_1$.
Now let $w$ be an integer that can be partitioned into a sum of  $u_1$  $1$'s and  $u-u_1$  $m$'s with $u_1<m$.
Applying Theorem~\ref{MCDQLS-SDPC} and Lemma \ref{type disjoint}, if there exists an MCDQLS$(w)$ in which every entry has $m$ consecutive nonzereo components starting at a position whose index is congruent to $0$ modulo $m$, 
then an MCDQLS$(mn + w)$ exists.

For $v =3a\geq 3\cdot 57=171$ or $v\in\{141,147,153,159\}$, if $v$  can be written as $v=3a=mn+w=7\cdot 4s+w$ with $s = \lfloor \frac{v}{28} \rfloor - 1 \geq 4$ and $3 \nmid s$, then  we have that the odd integer $w$ is in $ [29,55]$ and $3 \nmid w$. Note that  when $w$ is patitioned into a sum of some $m$'s and $1$'s, the number of parts in the partition always can be  at most $13$, which is  less than or equal to $n-2=4s-2$. Otherwise,    $v$  will be written as $v=3a=mn+w=7\cdot 4(s-1)+(28+w')$ with $s = \lfloor \frac{v}{28} \rfloor - 1 \geq 5$ and $3 \mid s$, then  we have that the odd integer $w=28+w'$ is in $ \{61,67,73,79\}$. Note that for this case  when $w$ is patitioned into a sum of some $m$'s and $1$'s, the number of parts in the partition is at most $w-6\lfloor w/7\rfloor =13$, which is also less than or equal to $n-2=4(s-1)-2$. So,   an MCDQLS$(v)$ exists for $v =3a\geq 3\cdot 57=171$ or $v\in\{141,147,153,159\}$.

For $v = 3a$ with $3 \le a \le 55$ and $v \notin \{7,9,11,15,21,27,33,39,51,63,75\}\cup \{ 141,147,153,159\}$, the values $n, m, u, w$ can be chosen as shown in the following table when applying Theorem~\ref{MCDQLS-SDPC} and Lemma \ref{type disjoint}.  By Lemma \ref{disjoint MCQLS$(m)$+ MCIDQLS$(m+1;1)$}, the pairwise disjoint or m-disjoint conditions in Theorem~\ref{MCDQLS-SDPC} can be guaranteed. By Lemma \ref{DLS} and Lemma \ref{DLs-6,10,14}, a $\mathrm{DLS}(n)$ that has other $u$ pairwise disjoint transversals exists. Therefore, the corresponding MCDQLS$(v)$ exists. And  for $v=117=3\cdot 39=3^2\cdot 13$, taking $n=9$, $m=13$  in Lemma \ref{MCDQLS-SDPC2}, we obtain an $\operatorname{MCDQLS}(v)$  by Lemma \ref{disjoint MCIDQLS$(v)$} and Lemma \ref{DLS-existence}.

$$
\begin{tabular}{|c|c|c|c|}
\hline
$a$ & $v = n \times m + w$ & Integer partition of $w$ & $u\leq n-2$ \\
\hline
$3 \cdot 5$   & $4 \times 11 + 1$   &        & $u = 1 \leq 2$\\
\hline
$3 \cdot 11$  & $14 \times 7 + 1$   &       & $u = 1 \leq 10$\\
\hline
$3^2 \cdot 5$ & $16 \times 7 + 23$  & $23 = 7 \times 3 + 1 \times 2$ & $u = 5 \leq 14$ \\
\hline
$3^3$         & $4 \times 16 + 17$  & $17 = 16 \times 1 + 1$        & $u = 2 \leq 2$  \\
\hline
$5 \cdot 7$   & $8 \times 13 + 1$   &        & $u = 1 \leq 6$\\
\hline
$5 \cdot 11$  & $4 \times 41 + 1$   &        & $u = 1 \leq 2$\\
\hline
$7^2$         & $16 \times 7 + 35$  & $35 = 7 \times 5$    & $u = 5 \leq 14$ \\
\hline
$19$          & $8 \times 7 + 1$    &        & $u = 1 \leq 6$\\
\hline
$23$          & $4 \times 17 + 1$   &        & $u = 1 \leq 2$\\
\hline
$29$          & $10 \times 7 + 17$  & $17 = 7 \times 2 + 1 \times 3$ & $u = 5 \leq 6$ \\
\hline
$31$          & $4 \times 23 + 1$   &        & $u = 1 \leq 2$\\
\hline
$37$          & $10 \times 11 + 1$  &       & $u = 1 \leq 6$ \\
\hline
$41$          & $14 \times 7 + 25$  & $25 = 7 \times 3 + 1 \times 4$ & $u = 7 \leq 10$ \\
\hline
$43$          & $8 \times 16 + 1$  &  & $u = 1 \leq 6$ \\
\hline
\end{tabular}
$$

In summary, 
an MCDQLS$(v)$ exists for all odd integers $v\geq$ with $v\notin\{7,9,11,15,21,27,33,39,51,63,$ $75\}$. \qed 
\end{proof}

\begin{theorem}\label{MCDQLS-result-even}
  For any even integer \( v \geq 6 \), there   exists an MCDQLS$(v)$ except possibly when \( v \in \{6, 8, 10, 12, 14, 18, 20, 22, 24, 26, 30, 32, 34, 38,   40, 46, 48, 50, 58, 62, 74, 82, 94, 122\} \).

\end{theorem}
\begin{proof}
      If $4\mid v$, write $v=4m$. By Remark \ref{MCIDQLS-R1}, if $m\notin B=\{2, 3, 4, 5, 6, 8, 10, 12, 14, 18, 20, 24, 26, 30,   32, $ $38, 62\} $,    there exist infinitely many pairwise disjoint idempotent MCQLS$(m)$s. By Lemma \ref{DLS},   there exists a DLS$(4)$  containing $4$ disjoint transversals, including the main diagonal and the anti-diagonal. Taking $n=4$ and $u=w=0$, then an MCDQLS$(v)$ exists by applying Theorem \ref{MCDQLS-SDPC}.

If $v=4m'$ and
$m' \in \{14,18,26,30,32,38,62\}$, rewrite $v$ as $v=8m$ where $m \in \{7,9,13,15,16,19,31\}$. By Lemma \ref{disjoint MCIDQLS$(v)$},  there exist infinitely many pairwise disjoint idempodent MCQLS$(m)$. By Lemma \ref{DLS-existence},   there exists a DLS$(8)$.
Taking  $n=8$ and $u=w=0$, then an MCDQLS$(v)$ exists by applying Theorem \ref{MCDQLS-SDPC}. Similarly, when
$v=4\cdot 20=5\times 16$ or $v=4\cdot 24= 6\times 16$, by Lemma \ref{disjoint MCIDQLS$(v)$},  there exist infinitely many pairwise disjoint MCDQLS$(16)$ and by Lemma \ref{DLS-existence},   there exists an DLS$(5)$ and an DLS$(6)$. 
Then we obtain an $\operatorname{MCDQLS}(v)$  by Lemma \ref{MCDQLS-SDPC2}.

For the case when $v=2a$ with $a\geq 9 $ being an odd  composite number and $a\neq 9,15,25$,     $a$ has an odd factor $ c\geq 7$.  Then there exist an DLS$(v/c)$ by Lemma \ref{DLS-existence}, and there exist infinitely many pairwise disjoint idempodent MCQLS$(c)$s by Lemma \ref{disjoint MCIDQLS$(v)$}. Taking $n=v/c$, $m=c$  and $u=w=0$,  an MCDQLS$(v)$ exists by applying Theorem \ref{MCDQLS-SDPC}.

For the case when \( v = 2p \), with \( p \ge 3 \) an odd prime, we write \( v \) as
\(
v = 2p = 4s \times 7 + w = 28s + w,
\)
where
\(
s = \left\lfloor \frac{v}{28} \right\rfloor - 1
    = \left\lfloor \frac{p}{14} \right\rfloor - 1\geq 0,
\)
and \( w \) satisfies
\(
w \equiv 2 \pmod{4}, \ 30 \le w \le 54.
\) Further observation shows that $w$ cannot be equal to $42$; otherwise we would deduce the prime $p = 7$, contradicting $v = 14 < 42$.
Rewrite $v = 2p$ again according to the value of $w$ as follows:
\begin{equation*}
    v = 2p = 
    \begin{cases} 
        28(s-3) + w' = 4(s-3) \times 7 + (w+84), & \text{if\ } w=30, \\
        28(s-2) + w' = 4(s-2) \times 7 + (w+56), & \text{if\ } w \in \{34,  46\}, \\
        28(s-1) + w' = 4(s-2) \times 7 + (w+28), & \text{if\ } w \in \{ 38,  50\}, \\
        28s + w' = 4s \times 7 + w, & \text{if\ } w=54.
    \end{cases}
\end{equation*}
From the previous construction of the MCDQLS($w'$) with $w' \in \{54, 66, 78, 90, 102, 114\}$, we know that the resulting MCDQLS($w'$) satisfies the property that the sum of the components of each entry is nonzero, by Remark~\ref{rmk1}.
Then by Lemma \ref{many MCIDQLS$(m+h;h)$}, there exist infinitely many disjoint MCDQLS$(w')$s. Partitioning $w'$ into a sum of some $7$'s and $1$'s, the maximum number $u$ of parts is $w' - 6\lfloor w'/7\rfloor = 12$ for $w' \in \{54, 66, 78\}$, and $w' - 6\lfloor w'/7\rfloor = 18$ for $w' \in \{90, 102, 114\}$. Furthermore, when $p \geq 127$ or $p \in \{83, 89, 97, 103, 109\}$, we have
\(
u \leq \frac{v - w'}{7} - 2.
\) 
Thus, in Theorem~\ref{MCDQLS-SDPC}, by taking $n = \frac{v - w'}{7}$ (where $4 \mid n$), $m = 7$,   $u \leq w' - 6\lfloor w'/7\rfloor$ and $m_1=m_2=\dots =m_u=1$, we obtain an MCDQLS$(v)$.

Similarly, for $p \in \{43,53,59,67,71,73,79,101,107,113\}$, we write $v = 2p$ in the form $v = n \times m + w$, with the values listed in Table below, where
\(
w = u_1 \times m + (u - u_1), \  0 \le u_1 \le u \le n-2.
\)
Note that there exists a DLS$(n)$ which contains, apart from the main anti-diagonal, $u_{0}$ disjoint transversals.  Setting $m_1=m_2=\dots =m_{u-u_1}=1$ and $m_{u-u_1+1}=m_{u-u_1+2}=\dots =m_{u}=m$, we obtain an an MCDQLS$(v)$,  by Theorem~\ref{MCDQLS-SDPC}, Lemma \ref{type disjoint}  and Lemma \ref{disjoint MCQLS$(m)$+ MCIDQLS$(m+1;1)$+MCIDQLS$(2m;m)$}.  

\begin{table}[h]
\centering
\begin{tabular}{|c|c|c|c|}
\hline
$p$ & $v = n \times m + w$ & $w = u_1 \times m + (u - u_1)$ & $u \leq u_{0} (\leq n-2$ )\\
\hline
43  & $86 = 10 \times 7 + 16$   & $16 = 2 \times 7 + 2$     & $4 \leq 6$   \\
\hline
53  & $106 = 12 \times 7 + 22$   & $22 = 3 \times 7 + 1$     & $4 \leq 10$   \\
\hline
59  & $118 = 10 \times 9 + 28$  & $28 = 3 \times 9 + 1$     & $4 \leq 6$   \\
\hline
67  & $134 = 14 \times 7 + 36$  & $36 = 5 \times 7 + 1$     & $6 \leq 10$ \\
\hline
71  & $142 = 14 \times 9+ 16$  & $16 = 1 \times 9 + 7$    & $8 \leq 10$ \\
\hline
73  & $146 = 8 \times 13 + 42$  & $42 = 3 \times 13 + 3$    & $6 \leq 6$   \\
\hline
79  & $158 = 8 \times 13 + 54$  & $54 = 4 \times 13 + 2$    & $6 \leq 6$   \\
\hline
101 & $202 = 12 \times 11 + 70$ & $70 = 6 \times 11 + 4$    & $10 \leq 10$ \\
\hline
107 & $214 = 16 \times 9 + 70$  & $70 = 7 \times 9 + 7$     & $14 \leq 14$ \\
\hline
113 & $226 = 14 \times 13 + 44$  & $44 = 3 \times 13 + 5$     & $8 \leq 10$  \\
\hline
\end{tabular}
\end{table}

So when $v = 2p$, $p \geq 3$ and $v\not\in\{ 6, 10, 14, 22, 26, 34, 38, 46, 58, 62, 74, 82, 94,  122\}$, an $\mathrm{MCDQLS}(v)$ exists.

In summary, the conclusion is proved.\qed
    \end{proof}

\begin{remark}
The MCDQLSs obtained in the proof of Theorem \ref{MCDQLS-result-odd} and Theorem \ref{MCDQLS-result-even} satisfy that each entry has a nonzero sum of its components; thus we can also obtain infinitely many disjoint MCDQLSs for the corresponding orders.
\end{remark}

\section{Conclusion}

In this paper, by employing direct constructive methods based on  row-quantum Latin rectangle,  strong complete mapping 
together with recursive techniques such as product constructions and singular direct product constructions, 
we have almost completely determined the existence spectrum of MCDQLS$(n)$, 
leaving only some possible exceptional orders for further investigation. That is, our main result is as follows.

\begin{theorem}
For $n \in \{2, 3, 4, 5\}$, there does not exist an $\mathrm{MCDQLS}(n)$.
For any integer $n \geq 6$, there exists  an $\mathrm{MCDQLS}(n)$  except possibly when $n \in P_1\cup P_2$, 
where 
 \begin{align*}
        P_1 &= \{6,8,10,12,14,18,20,24,26,30,32,38,62\}, \\
        P_2 &= \{7, 9,11,15, 21,22, 27, 33, 34,39, 40,46,48,50,51, 58, 63,74, 75,82,94,122\}.
           \end{align*}
\end{theorem}

In fact, we also   determined the existence  spectrum  of idempotent  MCQLS$(n)$, i.e., for $n \in \{2, 3, 4, 5\}$, there does not exist an $\mathrm{MCQLS}(n)$;  and for any integer $n \geq 6$, there exists an idempotent $\mathrm{MCQLS}(n)$  except possibly when $n \in P_1$.

Another byproduct of this paper is that, parallel to the classical Latin square theory where the necessary and sufficient condition for the existence of PLS$(n)$ is $\gcd(n,6)=1$, we determine that an MCPQLS$(n)$ exists when $\gcd(n,6)=1$ and $n \notin \{5, 7, 11\}$. Here, an MCPQLS$(5)$ does not exist, and the existence of MCPQLS$(7)$, MCPQLS$(11)$, and MCPQLS$(n)$ with $\gcd(n,6)\neq 1$ remains to be determined.

At the end of the paper, we point out that one can also use an QLS$(n)$ with $n$ disjoint transversals to construct a pair of orthogonal quantum Latin squares of order $n$. The following definition of a pair of orthogonal quantum Latin squares is due to Goyeneche et al.~\cite{cite3}.
Suppose there exist two $\text{QLS}(n)$, $U = (\ket{u_{i,j}})$ and $V = (\ket{v_{i,j}})$,
$0 \le i,j \le n-1$. If
\[
\ket{u_{i,j}} \otimes \ket{v_{i,j}} : 0 \le i,j \le n-1 
\]
forms an orthonormal basis of $\mathcal{H}_n^{\otimes 2} = \mathcal{H}_n \otimes \mathcal{H}_n$, then $U$ and $V$ are called a pair of orthogonal quantum Latin squares of order $n$, denoted by $2$-$\text{MOQLS}(n)$. If at least one of the pair of orthogonal quantum Latin squares is non-classical, it is called a non-classical $2$-$\text{MOQLS}(n)$.

In Lemma \ref{MCPDQLS-odd}, the corresponding $\mathrm{MCQLS}(7)$ $A$  has $7$ disjoint transversals parallel to the main diagonal. Construct $B=(\ket{b_{i,j}})$, where $\ket{b_{i,i+j}} = \ket{j}, i, j \in \mathbb{Z}_7$. This is a classical $\mathrm{QLS}(n)$ in which the entries on the main diagonal and on the broken diagonals in the same direction are respectively identical. One can easily verify that $A$ and $B$ form a pair of orthogonal quantum Latin squares of order $n$. 
    \begin{theorem}\label{MOQLS7}
        There exists a non-classical $2$-${MOQLS}(7)$. 
    \end{theorem}
This answers the open problem raised in Ref.~\cite{cite38}, 
thereby completely resolving the existence problem of non-classical $2$-MOQLS$(n)$ 
by confirming the last remaining exceptional order in Theorem \ref{MOQLS7}.

\section*{Acknowledgements}
The authors would like to thank Prof.~L. Zhu for several helpful suggestions on this topic.

\section*{Appendix A}

\noindent \textbf{Claim $1$.}
 \( M \) is a  QLS$(mn+w)$.
\ 

\

\begin{proof}
As all entries of \( M \) are unit vectors, we only need to check that the rows and columns of \( M \) are mutually orthogonal. We verify the row orthogonality here; the column case is similar.
To prove that  any two entries \(\ket{\alpha} \) and \(\ket{\beta}\) in the same row  
of \(M\) are orthogonal, eight cases are considered.  
\paragraph{Case 1.} \(\alpha, \beta \in \ket{a_{i,j}} \otimes_{+} B_q\), $-2\leq q \leq n-1$.

\noindent Let \(\alpha = \ket{a_{i,j}} \otimes_{+} \ket{b_{q,k,l}}\), \(\beta = \ket{a_{i,j}} \otimes_{+} \ket{b_{q,k,l'}}\), where \(l \neq l'\). Since $B$ is an MCQLS\((m)\) or an  idempotent MCQLS\((m)\)   \((\ket{b_{q,k,l}}, \ket{b_{q,k,l'}}) = 0\), we have
\begin{align*}
    (\alpha, \beta) &= \left( \begin{pmatrix}
\ket{a_{i,j}} \otimes \ket{b_{q,k,l}} \\
\mathbf{0}_{w}
\end{pmatrix},
\begin{pmatrix}
\ket{a_{i,j}} \otimes \ket{b_{q,k,l'}} \\
\mathbf{0}_{w}
\end{pmatrix} \right) \\
&= (\ket{a_{i,j}} \otimes \ket{b_{q,k,l}}, \ket{a_{i,j}} \otimes \ket{b_{q,k,l'}}) \\
&=\braket{a_{i,j}|a_{i,j}}\braket{b_{q,k,l}|b_{q,k,l'}}\\
&= 0.
\end{align*}

\paragraph{Case 2.} \(\alpha, \beta \in \ket{a_{p,i,j}} \otimes_{\widehat{p}} C_{p}\), $1\leq p \leq u$.

\noindent Since $C_p$ is an MCIQLS\((m+m_p,m_p)\)  and \((\ket{c_{p,s,t}}, \ket{c_{p,s,t'}}) = 0\), we have

\[
            \begin{aligned}
                (\alpha, \beta) &= \left( \begin{pmatrix}
                    \ket{a_{p,i,j}} \otimes \ket{c_{p,s,t,[0,m-1]}}\\
                    \mathbf{0}_{m_1}\\ \vdots\\
                    \ket{c_{p,s,t,{[m,m+m_p-1]}}}\\ \vdots\\
                    \mathbf{0}_{m_u}
                \end{pmatrix}, \begin{pmatrix}
                    \ket{a_{p,i,j}} \otimes \ket{c_{p,s,t',[0,m-1]}}\\
                    \mathbf{0}_{m_1}\\ \vdots\\
                    \ket{c_{p,s,t',[m,m+m_p-1]}}\\ \vdots\\
                    \mathbf{0}_{m_u}
                \end{pmatrix} \right)\\
                &= \braket{a_{p,i,j}|a_{p,i,j}} \braket{c_{p,s,t,[0,m-1]}|c_{p,s,t',[0,m-1]}} + \braket{c_{p,s,t,[m,m+m_p-1]} | c_{p,s,t',[m,m+m_p-1]} }\\
                &= \braket{c_{p,s,t}|c_{p,s,t'}}\\
                &= 0.
            \end{aligned}
            \]

\paragraph{Case 3.} \(\alpha \in \ket{a_{i,j}} \otimes_{+} B_q, \ \beta \in \ket{a_{i,j'}} \otimes_{+} B_{q'}\), $-2\leq q, q' \leq n-1$, \(j \neq j'\).

\noindent Let \(\alpha = \ket{a_{i,j}} \otimes_{+} \ket{b_{q,k,l}}\), \(\beta = \ket{a_{i,j'}} \otimes_{+} \ket{b_{q',k,l'}}\). Since \(\braket{a_{i,j}|a_{i,j'}} = 0\), we have
\begin{align*}
    (\alpha, \beta)&= \left( \begin{pmatrix}
\ket{a_{i,j}} \otimes \ket{b_{q,k,l}} \\
\mathbf{0}_{w}
\end{pmatrix},
\begin{pmatrix}
\ket{a_{i,j'}} \otimes \ket{b_{q',k,l'}} \\
\mathbf{0}_{w}
\end{pmatrix} \right) \\ &= (\ket{a_{i,j}} \otimes \ket{b_{q,k,l}}, \ket{a_{i,j'}} \otimes \ket{b_{q',k,l'}}) \\
&=\braket{a_{i,j}|a_{i,j'}}\braket{b_{q,k,l}|b_{q',k,l'}}  \\
&= 0.
\end{align*}

\paragraph{Case 4.} 
            $\alpha \in \ket{a_{i,j}} \otimes_{+} B_q$,  $\beta \in \ket{a_{p,i,j'}} \otimes_{\widehat{p}} C_{p}$, $1\leq p \leq u$, $j \neq j'$.
            
\noindent Let \(\alpha = \ket{a_{i,j}} \otimes_{+} \ket{b_{q,k,l}}\), \(\beta = \ket{a_{p,i,j'}} \otimes_{\widehat{p}} \ket{c_{p,k,t}}\). Since \(\braket{a_{i,j}|a_{p,i,j'}} = 0\), we have
 \[
            \begin{aligned}
                (\alpha, \beta) &= \left( \begin{pmatrix}
                    \ket{a_{i,j}} \otimes \ket{b_{q,k,l}}\\  \mathbf{0}_{w}
                \end{pmatrix}, \begin{pmatrix}
                    \ket{a_{p,i,j'}} \otimes \ket{c_{p,k,t,[0,m-1]}}\\  \mathbf{0}_{m_1}\\ \vdots\\ \ket{c_{p,k,t,[m,m+m_p-1]}}\\ \vdots\\ \mathbf{0}_{m_u}
                \end{pmatrix} \right)\\
                &= \braket{a_{i,j}|a_{p,i,j'}} \braket{b_{q,k,l}|c_{p,k,t,[0,m-1]}} + 0\\
                &= 0.
            \end{aligned}
            \]

\paragraph{Case 5.} \(\alpha \in \ket{a_{p,i,j}} \otimes_{\widehat{p}} C_{p}\), \( \beta \in \ket{a_{p',i,j'}} \otimes_{\widehat{p'}} C_{p'}\), $1\leq p, p' \leq u$, \(p \neq p', j \neq j'\).

\noindent Let \(\alpha = \ket{a_{i,j}} \otimes_{\widehat{p}} \ket{c_{p,s,t}}\), \(\beta = \ket{a_{i,j'}} \otimes_{\widehat{p'}} \ket{c_{p',s,t'}}\). Since \(\braket{a_{i,j}|a_{i,j'}}= 0\), we have
\begin{align*}
    (\alpha, \beta) &= \left( \begin{pmatrix}
\ket{a_{i,j}} \otimes \ket{c_{p,s,t,[m-1]}} \\ \mathbf{0}_{m_{1}}\\ \vdots \\ \ket{c_{p,s,t,[m,m+m_p-1]}}\\ \vdots \\ \mathbf{0}_{m_{u}}
    \end{pmatrix},
\begin{pmatrix}
\ket{a_{i,j}} \otimes \ket{c_{p',s,t',[m-1]}}\\ \mathbf{0}_{m_{1}}\\ \vdots \\ \ket{c_{p',s,t',[m,m+m_{p'}-1]}}\\ \vdots \\ \mathbf{0}_{m_{u}}
    \end{pmatrix} \right) \\
&= \braket{a_{i,j}|a_{i,j}}\braket{c_{p,s,t,[0,m-1]}|c_{p,s,t',[0,m-1]}}\\
&= 0.
\end{align*}

\paragraph{Case 6.} \(\alpha \in \ket{a_{p,i,j}} \otimes_{\widehat{p}} C_{p2}\), \( \beta \in \ket{a_{p,i',j'}} \otimes_{\widehat{p}} C_{p2}\),  $i\neq i'$, $j\neq j'$.

\noindent Let \(\alpha = \ket{a_{p,i,j}} \otimes_{\widehat{p}} \ket{c_{p,s,t}}\), \(\beta = \ket{a_{p,s',t'}} \otimes_{\widehat{p}} \ket{c_{p,s,t'}}\). Since \(\braket{a_{p,i,j}|a_{p,i',j'}}= 0\) and
$\ket{c_{p,s,t,[m,m+m_p-1]}}=\ket{c_{p,s,t',[m,m+m_p-1]}}=\mathbf{0}_{m_p }$, we have
\begin{align*}
    (\alpha, \beta) &= \left( \begin{pmatrix}
\ket{a_{p,i,j}} \otimes \ket{c_{p,s,t,[0,m-1]}} \\ \mathbf{0}_{m_{1 }}\\ \vdots \\ \ket{c_{p,s,t,[m,m+m_p-1]}}\\ \vdots \\ \mathbf{0}_{m_{u}}
    \end{pmatrix},
\begin{pmatrix}
\ket{a_{p,i',j'}} \otimes \ket{c_{p,s,t',[0,m-1]}}\\ \mathbf{0}_{m_{1}}\\ \vdots \\ \ket{c_{p,s,t',[m,m+m_p-1]}}\\ \vdots \\ \mathbf{0}_{m_{u}}
    \end{pmatrix} \right) \\
&= \braket{a_{p,i,j} | a_{p,i',j'}}\braket{c_{p,s,t,[0,m-1]}|c_{p,s,t',[0,m-1]}}+ (\mathbf{0}_{m_p }, \mathbf{0}_{m_p })\\
&= 0.
\end{align*}

\paragraph{Case 7.} \(\alpha \in \ket{a_{p,i,j}} \otimes_{\widehat{p}} C_{p2}\), \( \beta \in  M_w\).

\noindent Let \(\alpha = \ket{a_{p,i,j}} \otimes_{\widehat{p}} \ket{c_{p,s,t}}\), \(\beta =
\begin{pmatrix}
\mathbf{0}_{mn} \\
\ket{d_{e,f}}
\end{pmatrix}\). Since \(\ket{c_{p,s,t,[m,m+m_p-1]}} = \mathbf{0}_{m_{p }}\), we have
\begin{align*}
    (\alpha, \beta) &= \left( \begin{pmatrix}
\ket{a_{p,i,j}} \otimes \ket{c_{p,s,t,[0,m-1]}} \\ \mathbf{0}_{m_{1 }}\\ \vdots \\ \ket{c_{p,s,t,[m,m+m_p-1]}}\\ \vdots \\ \mathbf{0}_{m_{u}}
    \end{pmatrix},
\begin{pmatrix}
\mathbf{0}_{mn} \\
\ket{d_{e,f}}
\end{pmatrix} \right) \\
&= 0.
\end{align*}

\paragraph{Case 8.} \(\alpha, \beta\in M_w\).

\noindent Let $\alpha =
\begin{pmatrix}
\mathbf{0}_{mn} \\
\ket{d_{e,f}}
\end{pmatrix}
, \beta = \begin{pmatrix}
\mathbf{0}_{mn} \\
\ket{d_{e,f'}}
\end{pmatrix}$, where $f\neq f'$. Since $\braket{d_{e,f}| d_{e,f'}} = 0$,  we have
\begin{align*}
    (\alpha, \beta) &= \left(
\begin{pmatrix}
\mathbf{0}_{mn} \\
\ket{d_{e,f}}
\end{pmatrix},
\begin{pmatrix}
\mathbf{0}_{mn} \\
\ket{d_{e,f'}}
\end{pmatrix} \right) \\
&=\braket{d_{e,f}| d_{e,f'}}\\
&= 0.
\end{align*}

In summary, \( M \) is a  QLS$(mn+w)$. 
\qed 
 \end{proof}

 \noindent \textbf{Claim $2$.}
 The cardinality of \( M \)  is the maximum.
\ 

\

\begin{proof}
To any  two entries $\alpha$ and $\beta$  of $M$ in distinct cells in the array, eight cases are considered.

\paragraph{Case 1.} \(\alpha, \beta \in \ket{a_{i,j}} \otimes_{+} B_q\), $-2\leq q \leq n-1$.

\noindent Let \(\alpha = \ket{a_{i,j}} \otimes_{+} \ket{b_{q,k,l}}\), \(\beta = \ket{a_{i,j}} \otimes_{+} \ket{b_{q,k',l'}}\), where \((k,l))\neq (k',l')\).  
Since the  cardinality of $B_q$ is the maximum and and the two unit vectors $\ket{b_{k,l}}$ and $\ket{b_{k',l'}}$ are  distinct, then by Lemma \ref{state vector}, $\ket{a_{i,j}} \otimes \ket{b_{k,l}}\neq \ket{a_{i,j}} \otimes \ket{b_{k,l'}}$. Thus $\alpha \neq \beta$.

\paragraph{Case 2.} \(\alpha, \beta \in \ket{a_{p,i,j}} \otimes_{\widehat{p}} C_{p}\), $1\leq p \leq u$.

\noindent Let
  \[
           \alpha = \ket{a_{p,i,j}} \otimes_{\widehat{p}} \ket{c_{p,s,t}} = \begin{pmatrix}
                   \ket{a_{p,i,j}} \otimes \ket{c_{p,s,t,[0,m-1]}}\\
                   \mathbf{0}_{m_1}\\ \vdots\\ \ket{c_{p,s,t,[m,m+m_p-1]}}\\ \vdots\\ \mathbf{0}_{m_u}
               \end{pmatrix}, 
           \beta = \ket{a_{p,i,j}} \otimes_{\widehat{p}} \ket{c_{p,s',t'}} = \begin{pmatrix}
                   \ket{a_{p,i,j}} \otimes \ket{c_{p,s',t',[0,m-1]}}\\
                   \mathbf{0}_{m_1}\\ \vdots\\ \ket{c_{p,s',t',[m,m+m_p-1]}}\\ \vdots\\ \mathbf{0}_{m_u}
               \end{pmatrix}
           \] where $(s,t) \neq (s',t')$. If $\alpha = \beta$, then there would exist a real number $\theta\in [0,2\pi)$ such that
           \[ \ket{a_{p,i,j}} \otimes \ket{c_{p,s,t,[0,m-1]}} = e^{i\theta} \ket{a_{p,i,j}} \otimes \ket{c_{p,s',t',[0,m-1]}},\] 
              \[  c_{p,s,t,[m,m+m_p-1]} = e^{i\theta} c_{p,s',t',[m,m+m_p-1]}.\]
We would have  $\ket{c_{p,s,t}} = e^{i\theta} \ket{c_{p,s',t'}}$, which contradicts the fact that $C_p$ is an  MCIQLS$(m+m_p, m_p)$.  Thus $\alpha \neq \beta$.

\paragraph{Case 3.} \(\alpha, \beta \in M_w\). 

\noindent Let $\alpha = \begin{pmatrix}
               \mathbf{0}_{mn}\\ \ket{d_{e,f}}
           \end{pmatrix}$, $\beta = \begin{pmatrix}
               \mathbf{0}_{mn}\\ \ket{d_{e',f'}}
           \end{pmatrix}$, where $(e,f) \neq (e',f')$.
 The distinctness of \(\ket{d_{e,f}}\) and \(\ket{d_{e',f'}}\) immediately implies  $\alpha \neq \beta$.

\paragraph{Case 4.} $\alpha \in \ket{a_{i,j}} \otimes_{+} B_{q}, \beta \in \ket{a_{i',j'}} \otimes_{+} B_{q'}$, $-2\leq q, q' \leq n-1$, $(i,j) \neq (i',j')$.

 \noindent Let \[
           \alpha = \ket{a_{i,j}} \otimes_{+} \ket{b_{q,k,l}} = \begin{pmatrix}
                   \ket{a_{i,j}} \otimes \ket{b_{q,k,l}}\\
                   \mathbf{0}_w
               \end{pmatrix},   
           \beta = \ket{a_{i',j'}} \otimes_{+} \ket{b_{q',k',l'}} = \begin{pmatrix}
                   \ket{a_{i',j'}} \otimes \ket{b_{q, k',l'}}\\
                   \mathbf{0}_w
               \end{pmatrix}.
           \]
          where $(i,j)\neq (i',j')$.
If $q \neq q'$, since $B_{-2},B_{-1}, B_{0},B_{1}, \dots B_{n-1}$ are pairwise disjoint,  the two unit vectors $\ket{b_{q,k,l}}$ and $\ket{b_{q',k',l'}}$ are  distinct. By Lemma \ref{state vector}, $\ket{a_{i,j}} \otimes \ket{b_{q,k,l}}\neq \ket{a_{i,j}} \otimes \ket{b_{q',k',l'}}$, and thus $\alpha \neq \beta$.
If $q = q'$ and $i \neq  i'$,       
since the cardinality of $B_q$ ($q= -1$ or $-2$) is the maximum and and the two unit vectors $\ket{a_{i,j}}$ and $\ket{a_{i',j'}}$ (on the main diagonal or anti-diagonal
 ) are  distinct, then by Lemma \ref{state vector}, $\ket{a_{i,j}} \otimes \ket{b_{q,k,l}}\neq \ket{a_{i',j'}} \otimes \ket{b_{q',k',l'}}$, and thus $\alpha \neq \beta$.
Similarly, if $q = q'$, $i = i'$, $j\neq j'$, the distinctness of \(\ket{a_{i,j}}\) and \(\ket{a_{i',j'}}\)  implies  $\alpha \neq \beta$.

\paragraph{Case 5.} $\alpha \in \ket{a_{i,j}} \otimes_{+} B_{q}, \beta \in \ket{a_{p,i',j'}} \otimes_{\widehat{p}} C_p$, $-2\leq q, q' \leq n-1$, $(i,j) \neq (i',j')$.

 \noindent Let \[
           \alpha = \ket{a_{i,j}} \otimes_{+} \ket{b_{q, k, l}} = \begin{pmatrix}
                   \ket{a_{i,j}} \otimes \ket{b_{q,k,l}}\\
                   \mathbf{0}_w
               \end{pmatrix},  
           \beta = \ket{a_{p,i',j'}} \otimes_{\widehat{p}} \ket{c_{p,s,t}} = \begin{pmatrix}
                   \ket{a_{p,i',j'}} \otimes \ket{c_{p,s,t,[0,m-1]}}\\
                   \mathbf{0}_{m_1}\\  \vdots\\ \ket{c_{p,s,t,[m,m+m_p-1]}}\\ \vdots\\ \mathbf{0}_{m_u}
               \end{pmatrix}.
           \]
           If $\alpha = \beta$, then we have
            \[ \ket{c_{p,s,t,[m,m+m_p-1]}} = \mathbf{0}_{m_p}, \] and
               \[  \ket{a_{i,j}} \otimes \ket{b_{q,k,l}}= \ket{a_{p,i',j'}} \otimes_{\widehat{p}} \ket{c_{p,s,t,[0, m-1]}}. \]
           Then by Lemma  \ref{state vector},  the latter contradicts the fact that $B_q$ and $C_p$ are $m$-disjoint and $\ket{b_{q,k,l}} \neq \ket{c_{p,s,t}}$.
      
\paragraph{Case 6.} $\alpha \in \ket{a_{i,j}} \otimes_{+} B_q$, $\beta \in M_w$, $-2\leq q \leq n-1$.
 \noindent Let \[                     
           \alpha = \ket{a_{i,j}} \otimes_{+} \ket{b_{q,k,l}} = \begin{pmatrix}
                   \ket{a_{i,j}} \otimes \ket{b_{q,k,l}}\\
                   \mathbf{0}_w
               \end{pmatrix},
           \beta = \begin{pmatrix}
                   \mathbf{0}_{mn}\\ \ket{d_{e,f}}
               \end{pmatrix}.           \]
 Since $\ket{d_{e,f}} \neq \mathbf{0}_w$, we have $\alpha \neq \beta$.

 \paragraph{Case 7.} $\alpha \in \ket{a_{p,i,j}} \otimes_{\widehat{p}} C_p,\beta \in \ket{a_{p',i',j'}} \otimes_{\widehat{p}} C_{p'}$, $1\leq p,p' \leq u$, $(i,j)\neq (i',j')$. 
 
 \noindent Let \[                     
           \alpha = \ket{a_{p,i,j}} \otimes_{\widehat{p}} \ket{c_{p,s,t}} = \begin{pmatrix}
                   \ket{a_{p,i,j}} \otimes \ket{c_{p,s,t,[0,m-1]}}\\
                   \mathbf{0}_{m_1}\\ \vdots\\ \ket{c_{p,s,t,[m,m+m_p-1]}}\\ \vdots\\ \mathbf{0}_{m_u}
               \end{pmatrix},\ 
           \beta = \ket{a_{p',i',j'}} \otimes_{\widehat{p}} \ket{c_{p',s',t'}} = \begin{pmatrix}
                   \ket{a_{p',i',j'}} \otimes \ket{c_{p',s',t',[0,m-1]}}\\
                   \mathbf{0}_{m_1}\\ \vdots\\ \ket{c_{p',s',t',[m,m+m_{p'}-1]}}\\ \vdots\\ \mathbf{0}_{m_u}
               \end{pmatrix}.
           \]
           Note that $ \ket{a_{p,i,j}},  \ket{a_{p',i',j'}}$ are both in the set $ \{\ket{0}, \ket{1}, \dots, \ket{n-1}\} \subset \mathcal{H}_n$. 
           Consequently  if $ \ket{a_{p,i,j}}\neq  \ket{a_{p',i',j'}}$ and, $\ket{c_{p,s,t,[0,m-1]}}\neq \mathbf{0}_{m} $ or $\ket{c_{p',s',t',[0,m-1]}}\neq \mathbf{0}_{m}$, then $\alpha \neq \beta$ as $\ket{a_{p,i,j}} \otimes \ket{c_{p,s,t,[0,m-1]}} \neq \ket{a_{p',i',j'}} \otimes \ket{c_{p',s',t',[0,m-1]}} $. 
           If $ \ket{a_{p,i,j}}\neq  \ket{a_{p',i',j'}}$, $\ket{c_{p,s,t,[0,m-1]}}= \ket{c_{p',s',t',[0,m-1]}}= \mathbf{0}_{m}$, we also have $\alpha \neq \beta$ as $\ket{c_{p,s,t}}$ and $\ket{c_{p',s',t'}}$ are distinct when $p = p'$, and are $m$-disjoint when $p \ne p'$.  Similarly, if $ \ket{a_{p,i,j}}= \ket{a_{p',i',j'}}$, $\alpha \neq \beta$.

 \paragraph{Case 8.}  $\alpha \in \ket{a_{p,i,j}} \otimes_{\widehat{p}} C_p$, $\beta \in M_w$, $1\leq p \leq u$.

 \noindent Let \[
           \alpha = \ket{a_{p,i,j}} \otimes_{\widehat{p}} \ket{c_{p,s,t}} = \begin{pmatrix}
                   \ket{a_{p,i,j}} \otimes \ket{c_{p,s,t,[0,m-1]}}\\
                   \mathbf{0}_{m_1}\\ \vdots\\ \ket{c_{p,s,t,[m,m+m_p-1]}}\\ \vdots\\ \mathbf{0}_{m_u}
               \end{pmatrix}, 
            \beta = \begin{pmatrix}
                   \mathbf{0}_{mn}\\ \ket{d_{e,f}}
               \end{pmatrix}.         
           \]
         Note that  $\ket{a_{p,i,j}}\in \{\ket{0}, \ket{1}, \dots, \ket{n-1}\}$. As $D$ and $C_p$ are  $(m,p;m_1, \dots, m_p, \dots, m_u)$-type disjoint, the two unit vectors
                 \[ \begin{pmatrix}
                    \ket{c_{p,s,t,[0,m-1]}}\\
                   \mathbf{0}_{m_1}\\ \vdots\\ \ket{c_{p,s,t,[m,m+m_p-1]}}\\ \vdots\\ \mathbf{0}_{m_u}
               \end{pmatrix},  \  \begin{pmatrix}
                   \mathbf{0}_{m}\\ \ket{d_{e,f}}
               \end{pmatrix}\] are distinct.  Thus $\alpha \neq \beta$.

    In summary, the cardinality of \( M \)  is the maximum.        
\qed

\end{proof}

\section*{Appendix B}

The desired $\mathrm{DLS}(n)$s for Lemma~\ref{DLs-6,10,14} with $n\in \{6,10,14\}$ are given as follows. Among them, the $n-4$ transversals that are neither on the main diagonal nor on the main anti-diagonal are displayed in different colors.
Moreover, the \(\mathrm{DLS}(14)\) here comes from the Ref.~\cite{cite}.
\[
\begin{array}{cccccc}
   \textbf{0} & \textcolor{red}{1} & 2 & 3 & \textcolor{blue}{4} & \textbf{5}\\
   1 & \textbf{2} & \textcolor{blue}{0} & \textcolor{red}{5} & \textbf{3} & 4\\
   4 & \textcolor{blue}{3} & \textbf{5} & \textbf{0} & \textcolor{red}{2} & 1\\
   \textcolor{red}{3} & 0 & \textbf{1} & \textbf{4} & 5 & \textcolor{blue}{2}\\
   \textcolor{blue}{5} & \textbf{4} & 3 & 2 & \textbf{1} & \textcolor{red}{0}\\
   \textbf{2} & 5 & \textcolor{red}{4} & \textcolor{blue}{1} & 0 & \textbf{3}\\
\end{array}
\quad \quad \quad 
\begin{array}{*{10}{c}}
\textbf{2} & \textcolor{blue}{9} & \textcolor{brown}{6} & 7 & \textcolor{red}{1} & \textcolor{yellow}{0} & \textcolor{cyan}{4} & \textcolor{green}{8} & 5 & \textbf{3}\\
\textcolor{cyan}{3} & \textbf{7} & \textcolor{blue}{5} & \textcolor{green}{1} & \textcolor{brown}{0} & 6 & 2 & \textcolor{yellow}{9} & \textbf{8} & \textcolor{red}{4}\\
\textcolor{yellow}{5} & \textcolor{red}{6} & \textbf{0} & \textcolor{blue}{3} & 8 & \textcolor{cyan}{7} & \textcolor{brown}{1} & \textbf{4} & 9 & \textcolor{green}{2}\\
\textcolor{blue}{8} & 1 & \textcolor{green}{9} & \textbf{4} & \textcolor{yellow}{2} & \textcolor{red}{5} & \textbf{0} & \textcolor{cyan}{6} & \textcolor{brown}{3} & 7\\
\textcolor{red}{7} & 8 & \textcolor{cyan}{2} & 0 & \textbf{3} & \textbf{9} & \textcolor{green}{6} & \textcolor{blue}{1} & \textcolor{yellow}{4} & \textcolor{brown}{5}\\
\textcolor{green}{0} & \textcolor{brown}{4} & \textcolor{red}{8} & \textcolor{yellow}{6} & \textbf{7} & \textbf{1} & 5 & 3 & \textcolor{blue}{2} & \textcolor{cyan}{9}\\
4 & \textcolor{cyan}{0} & 3 & \textbf{5} & \textcolor{blue}{6} & \textcolor{brown}{8} & \textbf{9} & \textcolor{red}{2} & \textcolor{green}{7} & \textcolor{yellow}{1}\\
\textcolor{brown}{9} & \textcolor{yellow}{3} & \textbf{1} & \textcolor{cyan}{8} & \textcolor{green}{4} & 2 & \textcolor{blue}{7} & \textbf{5} & \textcolor{red}{0} & 6\\
1 & \textbf{2} & 4 & \textcolor{red}{9} & \textcolor{cyan}{5} & \textcolor{green}{3} & \textcolor{yellow}{8} & \textcolor{brown}{7} & \textbf{6} & \textcolor{blue}{0}\\
\textbf{6} & \textcolor{green}{5} & \textcolor{yellow}{7} & \textcolor{brown}{2} & 9 & \textcolor{blue}{4} & \textcolor{red}{3} & 0 & \textcolor{cyan}{1} & \textbf{8}\\
\end{array}
\]

\vspace{1em}
\[
\centering
\begin{array}{*{14}{c}}
  \textbf{3} & \textcolor{blue}{11} & \textcolor{brown}{9} & \textcolor{red}{4} & \textcolor{yellow}{0} & \textcolor{green}{10} & 2 & \textcolor{cyan}{7} & \textcolor{purple}{12} & \textcolor{orange}{8} & \textcolor{teal}{13} & \textcolor{pink}{5} & 1 & \textbf{6}\\ 
  4 & \textbf{5} & \textcolor{blue}{13} & \textcolor{brown}{11} & \textcolor{red}{6} & \textcolor{yellow}{2} & \textcolor{green}{12} & \textcolor{purple}{0} & \textcolor{orange}{10} & \textcolor{teal}{1} & \textcolor{pink}{7} & 3 & \textbf{8} & \textcolor{cyan}{9}\\ 
  \textcolor{green}{0} & 6 & \textbf{7} & \textcolor{blue}{1} & \textcolor{brown}{13} & \textcolor{red}{8} & \textcolor{yellow}{4} & \textcolor{orange}{12} & \textcolor{teal}{3} & \textcolor{pink}{9} & 5 & \textcolor{red}{10} & \textcolor{cyan}{11} & \textcolor{purple}{2}\\ 
  \textcolor{yellow}{6} & \textcolor{green}{2} & 8 & \textbf{9} & \textcolor{blue}{3} & \textcolor{brown}{1} & \textcolor{red}{10} & \textcolor{teal}{5} & \textcolor{pink}{11} & 7 & \textbf{12} & \textcolor{cyan}{13} & \textcolor{purple}{4} & \textcolor{orange}{0}\\ 
  \textcolor{red}{12} & \textcolor{yellow}{8} & \textcolor{green}{4} & 10 & \textbf{11} & \textcolor{blue}{5} & \textcolor{brown}{3} & \textcolor{pink}{13} & 9 & \textbf{0} & \textcolor{cyan}{1} & \textcolor{purple}{6} & \textcolor{orange}{2} & \textcolor{teal}{7}\\ 
  \textcolor{brown}{5} & \textcolor{red}{0} & \textcolor{yellow}{10} & \textcolor{green}{6} & 12 & \textbf{13} & \textcolor{blue}{7} & 11 & \textbf{2} & \textcolor{cyan}{3} & \textcolor{purple}{8} & \textcolor{orange}{4} & \textcolor{teal}{9} & \textcolor{pink}{1}\\ 
  \textcolor{blue}{9} & \textcolor{brown}{7} & \textcolor{red}{2} & \textcolor{yellow}{12} & \textcolor{green}{8} & 0 & \textbf{1} & \textbf{4} & \textcolor{cyan}{5} & \textcolor{purple}{10} & \textcolor{orange}{6} & \textcolor{teal}{11} & \textcolor{pink}{3} & 13\\ 
  1 & \textcolor{teal}{10} & \textcolor{orange}{3} & \textcolor{purple}{13} & \textcolor{pink}{4} & \textcolor{cyan}{12} & \textbf{9} & \textbf{2} & 6 & \textcolor{green}{5} & \textcolor{yellow}{11} & \textcolor{red}{7} & \textcolor{brown}{0} & \textcolor{blue}{8}\\ 
  \textcolor{teal}{8} & \textcolor{orange}{1} & \textcolor{purple}{11} & \textcolor{pink}{2} & \textcolor{cyan}{10} & \textbf{7} & 13 & \textcolor{blue}{6} & \textbf{0} & 4 & \textcolor{green}{3} & \textcolor{yellow}{9} & \textcolor{red}{5} & \textcolor{brown}{12}\\ 
  \textcolor{orange}{13} & \textcolor{purple}{9} & \textcolor{pink}{0} & \textcolor{cyan}{8} & \textbf{5} & 11 & \textcolor{teal}{6} & \textcolor{brown}{10} & \textcolor{blue}{4} & \textbf{12} & 2 & \textcolor{green}{1} & \textcolor{yellow}{7} & \textcolor{red}{3}\\ 
  \textcolor{purple}{7} & \textcolor{pink}{12} & \textcolor{cyan}{6} & \textbf{3} & 9 & \textcolor{teal}{4} & \textcolor{orange}{11} & \textcolor{red}{1} & \textcolor{brown}{8} & \textcolor{blue}{2} & \textbf{10} & 0 & \textcolor{green}{13} & \textcolor{yellow}{5}\\ 
  \textcolor{pink}{10} & \textcolor{cyan}{4} & \textbf{1} & 7 & \textcolor{teal}{2} & \textcolor{orange}{9} & \textcolor{purple}{5} & \textcolor{yellow}{3} & \textcolor{red}{13} & \textcolor{brown}{6} & \textcolor{blue}{0} & \textbf{8} & 12 & \textcolor{green}{11}\\ 
  \textcolor{cyan}{2} & \textbf{13} & 5 & \textcolor{teal}{0} & \textcolor{orange}{7} & \textcolor{purple}{3} & \textcolor{pink}{8} & \textcolor{green}{9} & \textcolor{yellow}{1} & \textcolor{red}{11} & \textcolor{brown}{4} & \textcolor{blue}{12} & \textbf{6} & 10\\ 
  \textbf{11} & 3 & \textcolor{teal}{12} & \textcolor{orange}{5} & \textcolor{purple}{1} & \textcolor{pink}{6} & \textcolor{cyan}{0} & 8 & \textcolor{green}{7} & \textcolor{yellow}{13} & \textcolor{red}{9} & \textcolor{brown}{2} & \textcolor{blue}{10} & \textbf{4}\\ 
\end{array}
\]

\end{document}